\documentclass{amsart}

\usepackage[utf8]{inputenc}
\usepackage[T2A,T1]{fontenc}

\usepackage{lmodern}
\usepackage{amssymb}
\usepackage{amsmath}
\usepackage{mathrsfs}
\usepackage{stmaryrd}
\usepackage{enumitem}
\usepackage[olditem,oldenum]{paralist}

\usepackage{graphics}

\usepackage[all,2cell]{xy} 
\UseAllTwocells

\usepackage{dsfont}

\DeclareSymbolFont{EulerScripta}{U}{euf}{m}{n}
\SetSymbolFont{EulerScripta}{bold}{U}{euf}{b}{n} 
\DeclareSymbolFontAlphabet\matheufm{EulerScripta}

\DeclareSymbolFont{EulerScriptb}{U}{eur}{m}{n}
\SetSymbolFont{EulerScriptb}{bold}{U}{eur}{b}{n} 
\DeclareSymbolFontAlphabet\matheurm{EulerScriptb}

\DeclareSymbolFont{EulerScriptc}{U}{eus}{m}{n}
\SetSymbolFont{EulerScriptc}{bold}{U}{eus}{b}{n} 
\DeclareSymbolFontAlphabet\matheusm{EulerScriptc}
\newcommand\eusm{\matheusm}

\usepackage{xspace}

\usepackage{comment}

\newcommand{\theoremref}[1]{\hyperref[#1]{Theorem~\ref*{#1}}}
\newcommand{\lemmaref}[1]{\hyperref[#1]{Lemma~\ref*{#1}}}
\newcommand{\remarkref}[1]{\hyperref[#1]{Remark~\ref*{#1}}}
\newcommand{\definitionref}[1]{\hyperref[#1]{Definition~\ref*{#1}}}
\newcommand{\propositionref}[1]{\hyperref[#1]{Proposition~\ref*{#1}}}
\newcommand{\conjectureref}[1]{\hyperref[#1]{Conjecture~\ref*{#1}}}
\newcommand{\corollaryref}[1]{\hyperref[#1]{Corollary~\ref*{#1}}}
\newcommand{\exampleref}[1]{\hyperref[#1]{Example~\ref*{#1}}}
\newcommand{\exerciseref}[1]{\hyperref[#1]{Exercise~\ref*{#1}}}
\renewcommand{\eqref}[1]{\hyperref[#1]{(\ref*{#1})}}
\newcommand{\pararef}[1]{\hyperref[#1]{\S\ref*{#1}}}
\newcommand{\enumref}[1]{\hyperref[#1]{{\itshape{(\ref*{#1})}}}}
\newcommand{\propositionitemref}[2]{\hyperref[#1]{Proposition~\ref*{#1}-{\itshape{(\ref*{#2})}}}}

\theoremstyle{plain}
\newtheorem{theo}{Theorem}

\newtheorem{prop}[theo]{Proposition}
\newtheorem{lemm}[theo]{Lemma}
\newtheorem{coro}[theo]{Corollary}

\newtheorem*{theo*}{Theorem}

\theoremstyle{definition}
\newtheorem{defi}[theo]{Definition}

\theoremstyle{remark}
\newtheorem{rema}[theo]{Remark}

\newtheorem*{conv}{Conventions}

\renewcommand{\ge}{\geqslant}
\renewcommand{\le}{\leqslant}

\newcommand{\oMCrv}{\overline{\mathrm{M}}\Crv}
\newcommand{\uMCrv}{\underline{\mathrm{M}}\Crv}
\newcommand{\dR}{\mathrm{dR}}
\newcommand{\red}{\mathrm{red}}
\newcommand{\bbG}{\mathbf{G}}
\newcommand{\Sing}{\mathrm{Sing}}
\newcommand{\xra}{\xrightarrow}
\newcommand{\op}{\mathrm{op}}
\newcommand{\scr}{\mathscr}
\newcommand{\ra}{\rightarrow}
\newcommand{\Id}{\mathrm{Id}}
\newcommand{\one}{\mathds{1}}
\newcommand{\Hom}{\mathrm{Hom}}
\newcommand{\Isheaf}{\matheusm{I}}
\newcommand{\Jsheaf}{\matheusm{J}}
\newcommand{\Osheaf}{\mathscr{O}}
\newcommand{\bbQ}{\mathbb{Q}}

\newcommand{\bbZ}{\mathbb{Z}}
\newcommand{\bbC}{\mathbb{C}}
\newcommand{\bbK}{K}
\newcommand{\bbH}{\mathbb{H}}
\newcommand{\Proj}{\mathop{\mathrm{Proj}}}
\newcommand{\Spec}{\mathop{\mathrm{Spec}}}

\newcommand{\Gr}{\mathrm{Gr}}
\newcommand{\Ker}{\mathop{\mathrm{Ker}}\nolimits}
\newcommand{\Coker}{\mathop{\mathrm{Coker}}\nolimits}

\newcommand{\colim}{\mathop{\mathrm{colim}}}
\newcommand{\an}{\mathrm{an}}
\renewcommand{\mod}{\mathop{\mathsf{mod}}\nolimits}
\newcommand{\comod}{\mathop{\mathsf{comod}}\nolimits}

\newcommand{\sM}{\mathscr{M}}

\newcommand{\sA}{\mathscr A}
\newcommand{\sB}{\mathscr B}
\newcommand{\sP}{\mathscr P}

\newcommand{\sS}{\mathscr S}
\newcommand{\End}{\mathrm{End}}
\newcommand{\name}[1]{{\scshape{#1}}}
\newcommand{\id}{\mathrm{id}}
\newcommand{\bH}{\mathbf{H}}
\newcommand{\vc}{\star}
\newcommand{\Aff}{\mathrm{Aff}}

\newcommand{\gr}{\mathrm{Gr}}

\newcommand{\EHM}{\mathrm{EHM}}

\newcommand{\ECM}{\mathrm{ECM}}
\newcommand{\ECMM}{\mathrm{ECMM}}

\newcommand{\et}{\mathrm{\'et}}
\newcommand{\drR}{\mathsf{R}}

\newcommand{\sT}{\mathrm{T}}

\newcommand{\Crv}{\mathrm{Crv}}
\newcommand{\MCrv}{\mathrm{MCrv}}
\newcommand{\Betti}{\mathrm{B}}

\newcommand{\ab}{\mathrm{ab}}
\newcommand{\uni}{\mathrm{uni}}
\newcommand{\mul}{\mathrm{mul}}
\newcommand{\Ext}{\mathrm{Ext}}
\newcommand{\Pic}{\mathrm{Pic}}
\newcommand{\Div}{\mathrm{Div}}
\newcommand{\Lie}{\mathop{\mathrm{Lie}}\nolimits}
\newcommand{\Sw}{\mathrm{Sw}}
\newcommand{\cl}{\mathrm{cl}}
\newcommand{\Ksheaf}{\mathscr{K}}
\newcommand{\LM}{\mathsf{LM}}
\newcommand{\oLM}{\overline{\mathsf{L}}\mathsf{M}}
\newcommand{\uLM}{\underline{\mathsf{L}}\mathsf{M}}

\newcommand{\Del}{\mathrm{Del}}
\newcommand{\fil}{\mathrm{fil}}
\newcommand{\sa}{\mathrm{sa}}

\newcommand{\oMPo}{\overline{\mathrm{M}}\mathrm{Po}}
\newcommand{\uMPo}{\underline{\mathrm{M}}\mathrm{Po}}
\newcommand{\bbP}{\mathbb{P}}
\newcommand{\UNM}{\mathrm{UNM}}
\newcommand{\INM}{\mathrm{INM}}
\newcommand{\bLM}{\mathbf{LM}}

\newcommand{\tunMot}{{}^t\kern-3pt\mathscr M}

\makeatletter
\let\@@seccntformat\@seccntformat
\renewcommand*{\@seccntformat}[1]{%
  \expandafter\ifx\csname @seccntformat@#1\endcsname\relax
    \expandafter\@@seccntformat
  \else
    \expandafter
      \csname @seccntformat@#1\expandafter\endcsname
  \fi
    {#1}%
}
\def\subsection{\@startsection{subsection}{2}%
  \z@{.5\linespacing\@plus.7\linespacing}{0em}%
  {\normalfont}}
\makeatother

\usepackage[unicode,bookmarks]{hyperref}
\usepackage[usenames,dvipsnames]{xcolor}
\hypersetup{colorlinks=true,citecolor=NavyBlue,linkcolor=BrickRed,urlcolor=Green}

\begin{document}

\title{Nori motives of curves with modulus and Laumon 1-motives}
\author{Florian Ivorra}
\address{Institut de recherche math\'ematique de Rennes\\ UMR 6625 du CNRS\\ Universit\'e de Rennes 1\\
Campus de Beaulieu\\
35042 Rennes cedex (France)}
\email{florian.ivorra@univ-rennes1.fr}

\author{Takao Yamazaki}
\address{Institute of mathematics \\ Tohoku University\\ Aoba, Senda\"i\\
980-8578 (Japan)}
\email{ytakao@math.tohoku.ac.jp}



\begin{abstract}
Let $k$ be a number field. We describe the category of Laumon 1-isomotives over $k$ as the universal category in the sense of Nori associated with a quiver representation built out of smooth proper $k$-curves with two disjoint effective divisors and a notion of $H^1_\dR$ for such  "curves with modulus".
This result extends and relies on the theorem of \name{J. Ayoub} and \name{L. Barbieri-Viale} that describes Deligne's category of 1-isomotives in terms of Nori's Abelian category of motives.
\end{abstract}
\thanks{The second author is supported by JSPS KAKENHI Grant (15K04773).}
\maketitle


\section{Introduction}
Let $k$ be a field of characteristic zero with an embedding $k\hookrightarrow\bbC$ into the field of complex numbers.
\subsection{}
Let $R$ be a field or a Dedekind ring and $T:\scr D\ra\mod(R)$ be a representation of a quiver $\scr D$ with values in the category $\mod(R)$ of finitely generated projective $R$-modules. In the unpublished work \cite{Nori} (see also \cite{LevineKHB,BookNori} for surveys), \name{M. Nori} has constructed an $R$-coalgebra $\eusm{C}_T$ such that the representation $T$ has a universal factorization (see \theoremref{NoriLength})
$$\scr D\xra{\overline{T}}\comod(\eusm{C}_T)\xra{F_T}\mod(R) $$
where $\comod(\eusm{C}_T) $ is the category of left $\eusm{C}_T$-comodules that are finitely generated over $R$, $\overline{T}$ is a representation and $F_T$ is the forgetful functor. 

Then, \name{M. Nori} applies this formalism to Betti homology to obtain the Abelian category $\EHM $ of effective homological motives over $k$ (see e.g. \cite{Nori,LevineKHB,BookNori}).  By construction, given a $k$-variety $X$, a closed (reduced) subscheme $Y\subseteq X$ and an integer $i\in\bbZ$, there is a motive $\overline{H}_i(X,Y)$ in $\EHM$ which realizes to the usual Betti homology.


\subsection{}\label{paraABV} 
\name{J. Ayoub} and \name{L. Barbieri-Viale}  have shown in \cite[Theorem 5.2, Theorem 6.1]{MR3302623} that the thick Abelian subcategory of Nori's category of effective homological motives generated by the $\overline{H}_0$ and $\overline{H}_1$ of pairs is equivalent to: \begin{inparaenum} \item[(a)] the Abelian category $\EHM_1$ associated with the representation
\begin{align*}
\Crv_k^\op & \ra \mod(\bbZ)\\
(C,Y) & \mapsto  H_1(C,Y)
\end{align*}
where $\Crv_k$ is the category of pairs $(C, Y)$ where $C$ is a smooth affine $k$-curve, $Y\subseteq C$ is a closed subset consisting of finitely many closed points and   $H_1(C,Y)$ is the first Betti homology group of the pair $(C,Y)$ ;
 \item[(b)] the Abelian category $\tunMot_1$ of Deligne's 1-motives with torsion \cite{MR2018930,MR0498552}\end{inparaenum}. \par
Note that by \cite[Th\'eor\`eme 3.4.1]{MR2102056} the derived category of Deligne's Abelian category of 1-isomotives $\sM_{1,\bbQ}$ is known to be equivalent to the thick triangulated subcategory of Voevodsky's category of geometrical effective motives with rational coefficients generated by motives of smooth $k$-curves.

\subsection{} Such a description is not possible integrally for the extension of the theory of 1-motives introduced by \name{G. Laumon} in \cite{Laumon} and studied in \cite{BVB,MR2385300,MR2563147,MR3095229}. Indeed the category of Laumon 1-motives with torsion $\tunMot_1^a$ of \cite{BVB} contains the category of infinitesimal formal $k$-groups as a full subcategory\footnote{The category of infinitesimal formal $k$-groups  is equivalent via the Lie Algebra with the category of finite dimensional $k$-vector spaces.}. In particular not all Hom groups in $\tunMot_1^a$ are finitely generated Abelian groups and therefore there cannot exist a quiver $\scr D$ and a representation $T:\scr D\ra\mod(\bbZ)$ such that $\tunMot_1^a$ is equivalent to $\comod(\eusm C_T)$.

If the field $k$ is not a number field, the same obstruction applies with rational coefficients. The Abelian category $\sM_{1,\bbQ}^a$ of Laumon $1$-isomotives still contains the category of infinitesimal formal $k$-groups as a full subcategory and therefore not all its Hom groups are finite dimensional $\bbQ$-vector spaces. Again this prevents the existence of a quiver $\scr D$ and a representation $T:\scr D\ra\mod(\bbQ)$ such that $\sM_{1,\bbQ}^a$ is equivalent to $\comod(\eusm C_T)$.

\subsection{} If $k$ is a number field, one may still hope for describing the Abelian category $\sM_{1,\bbQ}^a$ of Laumon $1$-isomotives over $k$ via Nori's tannakian formalism. The main result of this work is such a description in that case.



More precisely, let a $k$-curve with modulus be a triplet $(X,Y,Z)$ where $X$ is a smooth proper $k$-curve and $Y,Z$ are effective divisors on $X$ with disjoint supports. Define the de Rham cohomology of a such a $k$-curve with modulus as the finite dimensional $k$-vector space
$$\bH^1_\dR(X,Y,Z):=\bH^1(X,[\Isheaf_Y\ra \Isheaf_Z^{-1}\Omega^1_X])$$
where $\Isheaf_Y$ and $\Isheaf_Z$ are the ideals in $\Osheaf_X$ that define $Y$ and $Z$.  The $k$-curves with modulus define a category $\oMCrv_k$ 
for which a morphism $(X,Y,Z)\ra (X',Y',Z')$ is a morphism $f:X\ra X'$ of $k$-varieties such that 
(1) $Y\leqslant f^*Y'$, (2) $Z - Z_\red \ge f^*(Z'-Z_\red')$, and (3) $Z_\red \ge (f^* Z')_\red$.
If $k$ is a number field, by forgeting the $k$-linear structure, the de Rham cohomology of curves with modulus define a functor $$\bH^1_\dR :\oMCrv_k^\op\ra\mod(\bbQ)$$
with values in the category of finite dimensional $\bbQ$-vector spaces. Our main theorem is the following (see \theoremref{MainTheo2}).

\begin{theo}\label{MainTheo}
Let $k$ be a number field. The $\bbQ$-linear Abelian category associated with the representation of quiver 
\begin{align*}
\bH^1_\dR :\oMCrv_k^\op &\ra\mod(\bbQ)\\
 (X,Y,Z)& \mapsto \bH^1_\dR(X,Y,Z):=\bH^1(X,[\Isheaf_Y\ra \Isheaf_Z^{-1}\Omega^1_X])
\end{align*}
 is equivalent to the category $\sM^a_{1,\bbQ}$ of Laumon 1-isomotives over $k$.

\end{theo}

\theoremref{MainTheo} generalizes the equivalence between (a) and (b) recalled in \pararef{paraABV} and proved by \name{J. Ayoub} and \name{L. Barbieri-Viale} in \cite[Theorem 5.2]{MR3302623}.
Note that we do not provide any definition for a non homotopy invariant analog of the full category of Nori's motives of varieties (of arbitrary dimension) with modulus. %
Moreover in \cite{MR3302623} the main theorems are valid over any field of characteristic zero embedded into the complex numbers and admit also integral coefficient variants. Here we are not able to provide such generality. 
\footnote{In recent papers \cite{BCL, BP},
a new construction of the universal category 
without finite dimensionality assumption is introduced,
which would enable us to define $\ECMM_1$ for arbitrary subfield of $\bbC$.
Unfortunately, we then lose 
a description of the category as $\comod(\eusm C_T)$,
which is essential in the proof of our main result
(see \propositionref{AyoubBV}).
}
We leave this issue as future problems.

\begin{conv}
Throughout the paper we work over a base field $k$
with a fixed embedding $k \hookrightarrow \bbC$.
In \pararef{sect:DefMotivesWithModulus}, \pararef{sect:deRhamMotives} and from \pararef{sect:start-assuming-k-is-Q} onward,
we further assume that $k$ is a number field. 
For a $k$-scheme $X$, we denote by $\Omega_X^1$
the sheaf of K\"ahler differentials on $X$ relative to $k$.
If $Z$ is a closed subscheme of $X$,
we write $\Isheaf_Z \subset \Osheaf_X$ for the ideal sheaf of $Z$.
For a vector space $V$ over $k$,
we write $V^*$ for the $k$-linear dual of $V$.
Let $R$ be a ring and let $R'$ be an $R$-algebra.
For an $R$-linear Abelian category $\sA$,
we denote by $\sA \otimes_R R'$ its scalar extension.
This is a $R'$-linear Abelian category
having the same objects as $\sA$ and such that
\begin{equation}\label{eq:scalar-ext}
 \Hom_{\sA \otimes_R R'}(A, B) = \Hom_{\sA}(A, B)\otimes_R R'.
\end{equation}
\end{conv}

\section{Reminders on Nori's tannakian formalism}


\subsection{}
Let $K$ be a field. Following \cite[Chapitre II, \S4]{Gabriel}, recall that a $K$-linear Abelian category $\sP$ is said to be finite if it is Noetherian and Artinian i.e. $\sP$ is essentially small and any object in $\sP$ has finite length.
We shall say that $\sP$ is Hom finite if for any objects $P,Q$ in $\sP$ the $K$-vector space $\sP(P,Q)$ is finite dimensional. By \cite[Theorem 2.1]{IvorraPNM}, we have the following theorem.

\begin{theo}\label{NoriLength}
Let $\sP$ be a  $K$-linear Abelian category  which is finite and Hom finite, $\scr D$ be a quiver\footnote{Recall that a quiver is simply a directed graph.} and $T:\scr D\ra\sP$ be a representation of the quiver $\scr D$ with values in $\sP$. Then, there exist a $K$-linear Abelian category $\sA$, a representation $R:\scr D\ra\sA$, a  $K$-linear faithful exact functor $F:\sA\ra\sP$ and an invertible 2-morphism  $\alpha:F\circ R\ra T$ such that for every $K$-linear Abelian category $\sB$, every representation $S:\scr D\ra\sB$, every $K$-linear exact faithful functor $G:\sB\ra\sP$, and every invertible 2-morphism $\beta:G\circ S\ra T$ the following conditions are satisfied.
\begin{itemize}
\item There exist a $\bbK$-linear functor $H:\sA\ra\sB$ and two invertible $2$-morphisms
$$\gamma:H\circ R\xrightarrow{\simeq}S;\qquad \delta:G\circ H\xrightarrow{\simeq}F$$ 
such that 
$$\xymatrix{{G\circ H\circ R}\ar[r]^-{G\vc \gamma}\ar[d]^-{\delta\vc R} & {G\circ S}\ar[d]^-{\beta}\\
{F\circ R}\ar[r]^-{\alpha} & {T}} $$
is commutative. 
\item If  $H':\sA\ra\sB$ is a $\bbK$-linear functor and 
$$\gamma':H'\circ R\xrightarrow{\simeq}S;\qquad \delta':G\circ H'\xrightarrow{\simeq}F$$ 
are two invertible $2$-morphisms such that the square
$$\xymatrix{{G\circ H'\circ R}\ar[r]^-{G\vc\gamma'}\ar[d]^-{\delta'\vc R} & {G\circ S}\ar[d]^-{\beta}\\
{F\circ R}\ar[r]^-{\alpha} & {T}} $$
is commutative, then there exists a unique $2$-morphism $\theta:H\ra H'$ such that $\gamma'\circ(\theta\vc R)=\gamma$ and $\delta'\circ(G\vc\theta)=\delta$. 
\end{itemize}
\end{theo}

It will be useful to keep in mind the following remark.
\begin{rema}\label{RemaNoriCat}
When $\scr P=\mod(K)$ the previous theorem is due to \name{M. Nori}.  More precisely, let $\scr E$ be a full subquiver of $\scr D$ with finitely many objects and $\End_K(T|_\scr E)$ be the subring of 
$$\prod_{q\in\scr E}\End_K(T(q))$$
formed by the elements $e=(e_q)_{q\in\scr E}$ such that $e_q\circ T(m)=T(m)\circ e_p $ for every object $p\in\scr E$ and every morphism $m:p\ra q$ in $\scr D$. Then, its linear dual
$${\eusm C}_{T|_\scr E}:=\End_K(T|_\scr E)^*$$
is a coassociative, counitary $K$-coalgebra which is finite dimensional over $K$. We may then consider the $K$-linear Abelian category $\comod({\eusm C}_T)$
of finite dimensional left comodules over the coassociative and counitary $K$-coalgebra
$${\eusm C}_T:=\colim_{\scr E\subseteq\scr D}{\eusm C}_{T|_{\scr E}}$$
where the colimit is taken over full subquivers of $\scr D$ with finitely many objects.

For every object $p\in\scr D$ the finite dimensional $\bbK$-vector space $T(p)$ inherits a structure of left ${\eusm C}_T$-comodule. This provides a representation $\overline{T}:\scr D\ra\comod({\eusm C}_T)$ such that $T=F_T\circ\overline{T}$ where $F_T:\comod({\eusm C}_T)\ra\mod(\bbK)$ is the forgetful functor. The main result proved by \name{M. Nori} is that the uplet $(\comod(\eusm{C}_T),\overline{T},F_T,\id)$ satisfies the universal property of \theoremref{NoriLength} when $\scr P=\mod(K)$.

The general case is deduced from \name{Nori}'s result. Indeed, let $\scr P$ be a finite and Hom finite $K$-linear Abelian category and $T:\scr D\ra \scr P$ be a representation. A result (see \cite[Corollary 4.3]{IvorraPNM}) that can be easily deduced from \cite[5.1 Theorem, 5.8]{Takeuchi} assures the existence of a $K$-linear exact faithful functor $\omega:\scr P\ra \mod(K)$. Let $\scr A:=\comod(\eusm{C}_{\omega\circ T})$ and consider the associated representation
$$R:=\overline{\omega\circ T}:\scr D\ra \comod(\eusm{C}_{\omega\circ T})=:\scr A.$$
The universal property of $(\scr A,R,F_{\omega\circ T},\Id)$ applied to the uplet $(\scr P,T,\omega,\Id)$ provides a $K$-linear exact faithful functor $F:\scr A\ra \scr P$ and an invertible natural transformation $\alpha : F\circ R\ra T$. One checks then that the uplet $(\scr A,R,F,\alpha)$ satisfies the universal property stated in \theoremref{NoriLength}  (see \cite{IvorraPNM} for details).
\end{rema}


\subsection{} Let $\scr D$ be a quiver and $T:\scr D\ra\mod(K)$ be a representation. Let $(\sB,G,R,\beta)$ be an uplet where $\sB$ is 
a $K$-linear Abelian category, $S:\scr D\ra\sB$ is a representation, $G:\sB\ra\sP$ is a $K$-linear exact faithful functor, and $\beta:G\circ S\ra T$ is an invertible natural transformation. By the universal property of \theoremref{NoriLength}, 
there exist a $\bbK$-linear functor $H:\comod(\eusm{C}_T)\ra\sB$ and two invertible natural transformations
$$\gamma:H\circ \overline{T}\xrightarrow{\simeq}S;\qquad \delta:G\circ H\xrightarrow{\simeq}F_T$$ 
such that the square
$$\xymatrix{{G\circ H\circ \overline{T}}\ar[r]^-{G\vc \gamma}\ar[d]^-{\delta\vc R} & {G\circ S}\ar[d]^-{\beta}\\
{F_T\circ \overline{T}}\ar@{=}[r] & {T}} $$
is commutative (here we use the notations from \remarkref{RemaNoriCat}). In \cite[Proposition 2.1]{MR3302623}, \name{J. Ayoub} and \name{L. Barbieri-Viale} have given a criterion for the functor $H$ to be an equivalence. The proof of our main result relies on this criterion. 

\begin{prop}[Ayoub \& Barbieri-Viale \cite{MR3302623}]\label{AyoubBV}
Assume the following conditions.
\begin{enumerate}
\item\label{CondA} for every vertices $p,q\in\scr D$, there exist $p\sqcup q$ in $\scr D$ and edges $i:p\ra p\sqcup q$, $j:q\ra p\sqcup q$ such that the map
$$S(i)+S(j): S(p)\oplus S(q)\ra S(p\sqcup q)$$
is an isomorphism in $\sB$ ;
\item\label{CondB} every object in $\sB$ is a quotient of an object of the form $S(p)$ for some vertex $p\in\scr D$.
\item\label{CondC} for every map $S(p)\ra B$ in $\sB$ there exists a finite sub-quiver $\mathscr{E}\subseteq \scr D$ containing
$p$ such that
$$\Ker\{T(p) = G\circ S(p)\ra G(B)\}$$
is a sub-$\End(T|_\scr E)$-module of $T(p)$.
\end{enumerate}
Then the functor $H:\comod(\eusm{C}_T)\ra\sB$ is an equivalence of categories. 
\end{prop}

\subsection{} Let $\scr P_1$ and $\scr P_2$ be two finite and Hom finite $K$-linear Abelian categories. Let $\scr D_1$, $\scr D_2$ be quivers, $D:\scr D_1\ra\scr D_2$ be a morphism of quivers and
$$T_1:\scr D_1\ra\sP_1,\qquad T_2:\scr D_2\ra\sP_2$$
 be two representations.
Let $(\sA_1,F_1,R_1,\alpha_1)$ and $(\sA_2,F_2,R_2,\alpha_2) $ be uplets obtained by applying \theoremref{NoriLength} to the representations $T_1 $ and $T_2$ respectively.

The next proposition shows that certain exact functors can be lifted to universal categories (for a proof, see \cite[Proposition 6.6]{IvorraPNM}).

\begin{prop}\label{Fonc3}
Let $(\Phi,\phi)$ be a pair where $\Phi:\sP_1\ra\sP_2$ is an \emph{exact} $K$-linear functor and $\phi:\Phi\circ T_1\ra T_2\circ D$ is an isomorphism of representations.
There exist an exact functor $\Psi:\sA_1\ra\sA_2$, an invertible natural transformation $\rho:\Phi\circ F_1\ra F_2\circ \Psi $, and an isomorphism of representations $\varrho:\Psi\circ R_1\ra R_2\circ D$ such that
\begin{equation}\label{Dia1Fonc3}
\xymatrix@R=.4cm{{\Phi\circ F_1\circ R_1}\ar[r]^-{\Phi\star\alpha_1}\ar[dd]_-{\rho\star R_1} & {\Phi\circ T_1}\ar[rd]^-{\phi} & {}\\
{} & {} & {T_2\circ D}\\
{F_2\circ\Psi\circ R_1}\ar[r]_-{F_2\star\varrho} & {F_2\circ R_2\circ D}\ar[ru]_-{\alpha_2\star D} &{}}
\end{equation}
is commutative.
\end{prop}

\subsection{} In this work we will need to lift natural transformations as well. Let $D_1,D_2:\scr D_1\ra \scr D_2$ be morphism of quivers. Let $(\Phi_1,\phi_1)$, $(\Phi_2,\phi_2)$ be a pairs where $\Phi_1,\Phi_2:\sP_1\ra\sP_2$ are \emph{exact} $K$-linear functor and $\phi_1:\Phi_1\circ T_1\ra T_2\circ D_1$, $\phi_2: \Phi_2\circ T_1\ra T_2\circ D_2$  are isomorphisms of representations.

By \propositionref{Fonc3}, there exist exact functors $\Psi_1,\Psi_2:\sA_1\ra\sA_2$, invertible natural transformations $\rho_1:\Phi_1\circ F_1\ra F_2\circ \Psi_1 $,  $\rho_2:\Phi_2\circ F_1\ra F_2\circ \Psi_2 $, and isomorphisms of representations $\varrho_1:\Psi_1\circ R_1\ra R_2\circ D_1 $, $\varrho_2:\Psi_2\circ R_1\ra R_2\circ D_2 $ such that the corresponding diagrams as in \eqref{Dia1Fonc3} are commutative.
\begin{prop}\label{NatTransf}
Let $(\theta,\theta_D)$ be a pair where $\theta:\Phi_1\ra\Phi_2$ and $\theta_D:D_1\ra D_2$  are natural transformations such that the square
$$\xymatrix{{\Phi_1\circ T_1}\ar[r]^-{\phi_1}\ar[d]^-{\theta\vc T_1} & {T_2\circ D_1}\ar[d]^{T_2\vc\theta_D}\\
{\Phi_2\circ T_1}\ar[r]^-{\phi_2} & {T_2\circ D_2}} $$
is commutative. Then there exists one and only one natural transformation $\overline{\theta}: \Psi_1\ra\Psi_2$ that makes the squares
$$\xymatrix{{\Psi_1\circ R_1}\ar[r]^-{\varrho_1}\ar[d]^-{\overline{\theta}\vc R_1} & {R_2\circ D_1}\ar[d]^{R_2\vc\theta_D}\\
{\Psi_2\circ R_1}\ar[r]^-{\varrho_2} & {R_2\circ D_2}}
\qquad
\xymatrix{{\Phi_1\circ F_1}\ar[r]^-{\rho_1}\ar[d]^-{\theta\vc F_1} & {F_2\circ \Psi_1}\ar[d]^{F_2\vc\overline{\theta}}\\
{\Phi_2\circ F_1}\ar[r]^-{\rho_2} & {F_2\circ \Psi_2}}
$$
commutative. 
\end{prop}

\begin{proof}
Let $X$ be an object in $\scr A_1$. Let us sketch the construction of a morphism
$\overline{\theta}_X:\Psi_1(X)\ra\Psi_2(X)$ in $\scr A_2$ which makes the square
$$\xymatrix{{\Phi_1(F_1(X))}\ar[r]^-{\rho_{1,X}}\ar[d]^-{\theta_{F_1(X)}} & {F_2(\Psi_1(X))}\ar[d]^-{F_2(\overline{\theta}_X)}\\
{\Phi_2(F_1(X))}\ar[r]^-{\rho_{2,X}} & {F_2(\Psi_2(X))}} $$
commutative. Since $F_2$ is faithful, such a morphism is necessarily unique. When $X=R_1(p)$ for $p\in\scr D_1$, we define $\overline{\theta}_X$ to be the unique morphism that makes the square
$$\xymatrix{{\Psi_1(X)}\ar[r]^-{\varrho_{1,p}}\ar[d]^-{\overline{\theta}_X} & {R_2(D_1(p))}\ar[d]^-{R_2(\theta_{D,p})}\\
{\Psi_2(X)}\ar[r]^-{\varrho_{2,p}} & {R_2(D_2(p))}} $$
commutative. This defines also $\overline{\theta}_X $ when $X$ is a finite direct sum of such objects. Assume now the existence of an epimorphism  $s:Y\ra X$ in $\scr A_1$ where $Y$ is an object for which $\overline{\theta}_Y$ has been constructed.  It is then enough to check the existence of a factorization $$\xymatrix{{\Psi_1(Y)}\ar[r]^-{\Psi_1(s)}\ar[d]^-{\overline{\theta}_Y} & {\Psi_1(X)}\ar@{.>}[d]\ar[r] & {0}\\
{\Psi_2(Y)}\ar[r]^-{\Psi_2(s)} & {\Psi_2(X)}\ar[r] & {0.}} $$
 As the rows are exact, this amounts to checking that $\Psi_2(s)\circ\overline{\theta}_Y$ vanishes on the kernel of $\Psi_1(s)$. But this is true since it is after applying $F_2$ and $F_2$ is faithful.

Similarly, one shows the existence of $\overline{\theta}_X$ when $X$ is any subobject of an object $Y$ in $\scr A_1$ for which $\overline{\theta}_Y$ has already been constructed. 

This concludes the proof since by \cite[Proposition 7.1.16]{BookNori} every object in $\scr A_1$ is a subquotient of a finite direct sum of objects of the form $X=R_1(p)$ for $p\in\scr D_1$.
\end{proof}

\begin{rema}\label{RemaMono}
Note that since $F_2$ is a $K$-linear exact and faithful functor, if $\theta$ is a monomorphism (resp. epimorphism) then $\overline{\theta}$ is a monomorphism (resp. epimorphism). 
\end{rema}

\section{Nori motives of curves with modulus}

\subsection{} 
In this subsection, we collect some preliminary results
on cohomology of curves.

\begin{prop}\label{prop:funct-dif}
Let $f : C \to C'$ be a finite $k$-morphism
of smooth proper connected $k$-curves.
Let $D$ and $D'$ be effective divisors on $C$ and $C'$ respectively.
\begin{enumerate}
\item 
Suppose $D \leqslant f^*D'$. 
Then the canonical map
$\Osheaf_{C'} \to f_* \Osheaf_C$ induces
$\Isheaf_{D'} \to f_* \Isheaf_{D}$
and the trace map
$f_* \Omega^1_C \to \Omega^1_{C'}$ induces
$f_*(\Isheaf_D^{-1} \Omega^1_C) \to \Isheaf_{D'}^{-1}\Omega^1_{C'}$.
\item 
Suppose $D-D_\red \ge f^*(D'-D_\red')$ and $D_\red \ge (f^*D')_\red$. 
(The latter condition is equilavent to 
$f(C \setminus |D|) \subset f(C' \setminus |D'|)$).
Then the canonical map
$\Omega^1_{C'} \to f_* \Omega^1_C$ induces
$\Isheaf_{D'}^{-1} \Omega_{C'} \to f_*(\Isheaf_{D}^{-1}\Omega_{C})$
and the trace map
$f_* \Osheaf_{C} \to \Osheaf_{C'}$ induces
$f_* \Isheaf_{D} \to \Isheaf_{D'}$.
\end{enumerate}
(Recall that by our convention 
$k$ is a subfield of $\bbC$,
that
$\Omega_C^1$ is the sheaf of K\"ahler differentials on $C$ relative to $k$,
and that $\Isheaf_D$
is the ideal sheaf defining $D$.)
\end{prop}

\begin{proof}
This follows from the following elementary lemma. 
\end{proof}

\begin{lemm}\label{lem:trace-differential}
Let $K$ be a function field of one variable over $k$,
and let $R \subset K$ be a discrete valuation ring containing $k$.
Let $L$ be a finite extension of $K$
and let $S$ be the integral closure of $R$ in $L$.
Denote by $\mathfrak{m}$ the maximal ideal of $R$,
and by $\mathfrak{n}_1, \dots, \mathfrak{n}_r$ the maximal ideals of $S$.
Let $e_i \in \bbZ_{>0}$ be the ramification index of $\mathfrak{n}_i$.
Let $m, n_1, \dots, n_r \ge 1$ be integers
and put 
$\mathfrak{n}^n := \mathfrak{n}_1^{n_1} \cdots \mathfrak{n}_r^{n_r},$
$\mathfrak{n}^{-n} := \mathfrak{n}_1^{-n_1} \cdots \mathfrak{n}_r^{-n_r}.$
\begin{enumerate}
\item 
Suppose $n_i \le e_i m$ for all $i$. 
Then
the canonical map $K \to L$ 
sends $\mathfrak{m}^m$ to $\mathfrak{n}^n$,
and the trace map $\Omega_{L/k}^1 \to \Omega_{K/k}^1$
sends
$\mathfrak{n}^{-n} \Omega_{S/k}^1$ to
$\mathfrak{m}^{-m} \Omega_{R/k}^1$.
\item 
Suppose  $n_i-1 \ge e_i (m-1)$ for all $i$. Then
the canonical map $\Omega_{K/k}^1 \to \Omega_{L/k}^1$ 
sends
$\mathfrak{m}^{-m} \Omega_{R/k}$ 
to $\mathfrak{n}^{-n} \Omega_{S/k}$,
and the trace map $L \to K$ sends
$\mathfrak{n}^{n}$ to $\mathfrak{m}^m.$
\end{enumerate}
\end{lemm}
\begin{proof}
The last statement of (2) follows from
\cite[Chapter III, Propositions 7, 13]{serreCL}.
All other statements are elementary.
\end{proof}

\begin{prop}\label{prop:uv}
Let $C$ be a smooth proper curve over $k$
and let $D$ be an effective divisor on $C$.
We set
\[ U(C, D):= H^0(C, \Isheaf_{D_\red}/\Isheaf_D),
\quad
V(C, D):= H^0(C, \Isheaf_{D_\red}\Isheaf_D^{-1}/\Osheaf_C).
\]
Then the differential map induces isomorphisms
\begin{align*}
d : U(C, D) \overset{\cong}{\ra}
H^0(C, (\Osheaf_C/\Isheaf_D \Isheaf_{D_\red}^{-1}) \otimes \Omega_C^1),
\\
d : V(C, D) \overset{\cong}{\ra}
H^0(C, (\Isheaf_D^{-1}/\Isheaf_{D_\red}^{-1}) \otimes \Omega_C^1).
\end{align*}
\end{prop}
\begin{proof}
Write $D=\sum_{P \in |C|} n_P P$.
Then we have
\begin{align*}
&U(C, D) \cong \bigoplus_{P \in |D|} \mathfrak{m}_P/\mathfrak{m}_P^{n_P},
\\
&H^0(C, (\Osheaf/\Isheaf_D \Isheaf_{D_\red}^{-1}) \otimes \Omega_C^1)
\cong \bigoplus_{P \in |D|} \Omega^1_{C, P}/\mathfrak{m}_P^{n_P-1}\Omega^1_{C, P},
\end{align*}
where $\mathfrak{m}_P$ denotes the maximal ideal of
the local ring $\Osheaf_{C, P}$ of $C$ at $P$.
Thus the first statement follows from the bijectivity of
\[ d : \mathfrak{m}_P/\mathfrak{m}_P^{n_P} \to 
\Omega^1_{C, P}/\mathfrak{m}_P^{n_P-1}\Omega^1_{C, P},
\]
which is readily seen.
Similarly, we have
\begin{align*}
&V(C, D) \cong \bigoplus_{P \in |D|} \mathfrak{m}_P^{1-n_P}/\Osheaf_{C, P},
\\
&H^0(C, (\Isheaf_D^{-1}/\Isheaf_{D_\red}^{-1}) \otimes \Omega^1_C)
\cong \bigoplus_{P \in |D|} \mathfrak{m}_P^{-n_P} \Omega^1_{C, P}/\mathfrak{m}_P^{-1}\Omega^1_{C, P}.
\end{align*}
Thus the second statement follows from the bijectivity of
\[ d: \mathfrak{m}_P^{1-n_P}/\Osheaf_{C, P}
\to \mathfrak{m}_P^{-n_P}\Omega^1_{C, P}/\mathfrak{m}_P^{-1}\Omega^1_{C, P},
\]
which is readily seen.
\end{proof}

\begin{coro}\label{cor:dual-uv}
The two $k$-vector spaces $U(C, D)$ and $V(C, D)$ 
are canonically dual to each other.
\end{coro}
\begin{proof}
We may suppose $D$ is (effective and) non-trivial.
Then we get
\[ U(C, D)=\ker[H^1(C, \Isheaf_D) \to H^1(C, \Isheaf_{D_\red})] \]
from an exact sequence
$0 \to \Isheaf_D \to \Isheaf_{D_\red} \to \Isheaf_{D_\red}/\Isheaf_D \to 0$.
On the other hand,
another exact sequence
$0 \to \Isheaf_{D_\red}^{-1} \otimes \Omega^1_C
\to \Isheaf_{D}^{-1} \otimes \Omega^1_C
\to (\Isheaf_D^{-1}/\Isheaf_{D_\red}^{-1}) \otimes \Omega^1_C \to 0$
and the above proposition yield
\[ V(C, D)=\Coker[H^0(C, \Isheaf_{D_\red}^{-1} \Omega^1_C) 
\to H^0(C, \Isheaf_D^{-1} \Omega^1_C)].
\]
Now the corollary follows from the Serre duality.
\end{proof}

\begin{coro}\label{prop:pull-push-uv}
Let $(C, D)$ and $(C', D')$ be pairs of
a smooth proper $k$-curves and an effective divisor.
Let $f : C \to C'$ be a finite $k$-morphism.
The canonical map
$\Osheaf_{C'} \to f_* \Osheaf_{C}$ 
and the trace map
$f_* \Omega^1_{C} \to \Omega^1_{C'}$
induce the following functoriality:
\begin{enumerate}
\item 
If $D \leqslant f^*D'$, then we have 
\begin{equation*}
 f^* :U(C', D') \to U(C, D), 
\qquad
 f_* :V(C, D) \to V(C', D').
\end{equation*}
\item 
If $D-D_\red \geqslant f^*(D'-D_\red')$
and $D_\red \geq (f^* D')_\red$, then we have 
\begin{equation*}
 f^* :V(C', D') \to V(C, D), 
\qquad
 f_* :U(C, D) \to U(C', D'). 
\end{equation*}
\end{enumerate}
\end{coro}
\begin{proof}
Since $D \le f^*D'$ implies
$D_\red \le (f^* D')_\red \le f^*(D_\red')$,
this follows from 
\propositionref{prop:funct-dif} and \ref{prop:uv}.
\end{proof}

\subsection{}\label{sect:def-MCrv}
Let us denote by $\oMCrv$ the following category. An object in $\oMCrv$ is a triplet $(X,Y,Z)$ where $X$ is smooth proper $k$-curve and $Y,Z$ are effective divisors on $X$ such that $|Y|\cap |Z|=\emptyset$. A morphism 
$$(X,Y,Z)\ra (X',Y',Z') $$
in $\oMCrv$, is a morphism $f:X\ra X'$ of $k$-varieties 
such that 
(1) $Y\leqslant f^*Y'$, (2) $Z - Z_\red \ge f^*(Z'-Z_\red')$,
and (3) $Z_\red \ge (f^* Z')_\red$
(equivalently, $f(X \setminus |Z|) \subset f(X' \setminus |Z'|)$).
It then follows from 
\propositionref{prop:funct-dif} that
the canonical map $\Osheaf_{X'} \to f_* \Osheaf_{X}$ induces morphisms of sheaves
\begin{equation}\label{eq:funct-1}
\Isheaf_{Y'} \to f_* \Isheaf_Y
\quad \text{and} \quad
\Isheaf_{Z'}^{-1} \Omega_{X'}^1 \to f_*(\Isheaf_Z^{-1} \Omega_{X}^1).
\end{equation}



It will be useful to consider also the following variant: 
$\uMCrv$ is the category with the same objects as $\oMCrv$ but this times a morphism 
$$(X,Y,Z)\ra (X',Y',Z') $$
in $\uMCrv_k$ is a morphism $f:X\ra X'$ of $k$-varieties such that 
(1) $Y - Y_\red \ge f^*(Y' - Y_\red')$,
(2) $Y_\red \ge (f^* Y')_\red$, and
(3) $Z \le f^*Z'$.
Again it then follows from 
\propositionref{prop:funct-dif} that
the trace map $f_* \Osheaf_{X} \to \Osheaf_{X'}$ 
induces morphisms of sheaves
\begin{equation}\label{eq:funct-2}
f_* \Isheaf_{Y'} \to \Isheaf_Y
\quad \text{and} \quad
f_* (\Isheaf_{Z'}^{-1} \Omega_{X'}^1) \to \Isheaf_Z^{-1} \Omega_{X}^1.
\end{equation}



\begin{defi}\label{DefidR}
Let $(X,Y,Z)$ be an object in the category $\oMCrv$.
We define 
$$\bH^1_\dR(X,Y,Z):=\bH^1(X,[\Isheaf_Y\ra \Isheaf_Z^{-1}\Omega^1_X])$$
to be the first hypercohomology group of 
the complex of $\Osheaf_X$-modules $[\Isheaf_Y\ra \Isheaf_Z^{-1}\Omega^1_X]$,
where $\Isheaf_Y$ is placed in degree zero. 
This is a finite dimensional $k$-vector space.
By \eqref{eq:funct-1}, we obtain a functor
\begin{equation}\label{eq:functors-bh1}
\bH^1_\dR:\oMCrv^\op\ra\mod(k),
\end{equation}
where $\mod(k)$ is the category of finite dimensional $k$-vector spaces.
We also have a functor
\begin{equation}\label{eq:functors-bh2}
{}^t{\bH}^1_\dR:\uMCrv \ra\mod(k)
\end{equation}
which takes the same value on objects as ${\bH}^1_\dR$
but acts on morphisms via \eqref{eq:funct-2}.
\end{defi}

\subsection{}
In the following,
see \propositionref{prop:uv} for the definition of $U(X, Y)$
and $V(X, Z)$.

\begin{prop}\label{CanDec}
For any $(X, Y, Z) \in \oMCrv$,
there is a canonical decomposition
\begin{equation}\label{eq:hdr-dec}
\bH^1_\dR(X,Y,Z)
\cong
\bH^1_\dR(X,Y_\red,Z_\red) \oplus
U(X, Y) \oplus V(X, Z).
%
\end{equation}
Moreover, 
the decomposition \eqref{eq:hdr-dec} is functorial with respect to maps in $\oMCrv$.
\end{prop}
\begin{proof}
Since $U(C, D_\red)=V(C, D_\red)=0$ for 
a pair of a smooth proper $k$-curve $C$ and an effective divisor $D$,
we are reduced to showing
\[ 
\bH^1_\dR(X,Y,Z)
\cong
\bH^1_\dR(X,Y_\red,Z) \oplus U(X, Y)
\cong
\bH^1_\dR(X,Y,Z_\red) \oplus V(X, Z).
\]
To show the first isomorphism, 
we construct canonical maps
\[ a : \bH^1_\dR(X,Y_\red,Z) \to \bH^1_\dR(X,Y,Z),
\quad
   b : \bH^1_\dR(X,Y,Z) \to \bH^1_\dR(X,Y_\red,Z)
\]
such that $b \circ a = \id$ and $\ker(b) \cong U(X, Y)$.
For this we first note that
the map 
\[
[\Isheaf_{Y} \ra \Isheaf_{Y}\Isheaf_{Y_\red}^{-1}\Isheaf_{Z}^{-1}\Omega^1_X]
\to
[\Isheaf_{Y_\red} \ra \Isheaf_{Z}^{-1}\Omega^1_X]
\]
(induced by the inclusions
$\Isheaf_{Y} \subset \Isheaf_{Y_\red}$
and 
$\Isheaf_{Y}\Isheaf_{Y_\red}^{-1}\Isheaf_{Z}^{-1}
\subset \Isheaf_{Z}^{-1}$)
is a quasi-isomorphism
by \propositionref{prop:uv}.
Using this, we define $a$ to be  the composition
\begin{align*}
\bH^1_\dR(X,Y_\red,Z) 
&=\bH^1(X,[ \Isheaf_{Y_\red} \ra \Isheaf_{Z}^{-1}\Omega^1_X])
\\
&\overset{\cong}{\leftarrow}
\bH^1(X,[ \Isheaf_{Y} \ra \Isheaf_{Y}\Isheaf_{Y_\red}^{-1}
\Isheaf_{Z}^{-1}\Omega^1_X])
\\
&\to
\bH^1(X,[ \Isheaf_{Y} \ra \Isheaf_{Z}^{-1}\Omega^1_X])
= \bH^1_\dR(X,Y,Z),
\end{align*}
where the second map is induced by the inclusion
$
\Isheaf_{Y}\Isheaf_{Y_\red}^{-1} \Isheaf_{Z}^{-1}
\subset \Isheaf_Z^{-1}$.
Next, $b$ is given by 
\begin{align*}
\bH^1_\dR(X,Y,Z)
&=\bH^1(X,[ \Isheaf_{Y} \ra \Isheaf_{Z}^{-1}\Omega^1_X])
\\
&\to
\bH^1(X,[ \Isheaf_{Y_\red} \ra \Isheaf_{Z}^{-1}\Omega^1_X])
= 
\bH^1_\dR(X,Y_\red,Z)
\end{align*}
which is induced by the inclusion
$\Isheaf_{Y} \subset \Isheaf_{Y_\red}$.
It is obvious that the composition $b \circ a$ is the identity.
It is also clear from this construction that $\ker(b) \cong U(X, Y)$.
Note also that \propositionref{prop:uv} tells us that
$\Coker(a) \cong U(X, Y)$,
as it should be.

The second isomorphism
$\bH^1_\dR(X,Y,Z)
\cong
\bH^1_\dR(X,Y,Z_\red) \oplus V(X, Z)$
is constructed in a similar way.
We omit it.
\end{proof}

\begin{prop}\label{prop:dual-hdr}
For any $(X, Y, Z) \in \oMCrv$,
the two $k$-vector spaces $\bH^1_\dR(X,Y,Z)$ and $\bH^1_\dR(X,Z,Y)$
are canonically dual to each other.
\end{prop}
\begin{proof}
Apply \lemmaref{lem:dual-curve} with
$C^*=[\Isheaf_Y \to \Isheaf_Z^{-1} \Omega^1_X]$
and
$D^*=[\Isheaf_Z \to \Isheaf_Y^{-1} \Omega^1_X]$.
\end{proof}

%
%

\begin{lemm}\label{lem:dual-curve}
Let $C^*$ and $D^*$ be two complexes of sheaves of $k$-vector spaces
on $X$ such that
$C^i$ and $D^i$ are locally free $\Osheaf_X$-modules for all $i$
and that $C^i=D^i=0$ unless $i \not\in \{ 0, 1 \}$.
Let $\wedge : \mathrm{Tot}(C^* \otimes_k D^*) \to \Omega_X^\bullet$ 
be a map of complexes and suppose that it induces
$C^0 \cong \underline{\Hom}_{\Osheaf_X}(D^1, \Omega^1_X)$ and
$C^1 \cong \underline{\Hom}_{\Osheaf_X}(D^0, \Omega^1_X)$.
Then $\wedge$ induces a perfect duality between
$\bH^i(X, C^*)$ and $\bH^{2-i}(X, D^*)$ for all $i$.
\end{lemm}
\begin{proof}
This is reduced to the Serre duality
by an exact sequence
\begin{align*}
\dots 
\to H^{i-1}(X, C^1) 
\to \bH^{i}(X, C^*) 
\to H^{i}(X, C^0) 
\to H^{i}(X, C^1) \to\dots
\end{align*}
and a similar sequence for $D^*$.
\end{proof}

\subsection{}\label{sect:DefMotivesWithModulus}
The following definition introduces our main object of studies.


\begin{defi} Let $k$ be a number field. The category $\ECMM_{1}$ of effective cohomological isomotives of curves with modulus is the $\bbQ$-linear category associated with the representation
$$\bH^1_\dR:\oMCrv^\op\ra\mod(\bbQ).$$
\end{defi}

By construction the representation $\bH^1_\dR$ has a factorization 
$$\oMCrv^\op\xra{\overline{\bH}^1_\dR}\ECMM_{1}\xra{F^a_{\dR}}\mod(\bbQ) $$
into a representation $\overline{\bH}^1_\dR$ 
and a $\bbQ$-linear faithful exact functor $F^a_{\dR}$.


\subsection{}\label{sect:EHM-ECM} 
Let $\Crv$ be the category defined as follows (see \cite[\S5.1]{MR3302623}). An object is a pair $(C, Y)$ where $C$ is a smooth affine curve and $Y\subseteq C$ is a closed subset consisting of finitely many closed points.
A morphism $(C, Y) \to (C', Y')$ is given by 
a $k$-morphism $f : C \to C'$ such that $f(Y) \subset Y'$.

Recall that by definition (see \cite[\S5.1]{MR3302623}) 
the $\bbQ$-linear Abelian category $\EHM_1$ of effective homological isomotives of curves
\footnote{Note that in \cite{MR3302623} the category $\EHM_1 $ is denoted by ${\EHM''_1}  $ while $\EHM_1$ stands for the  the thick Abelian subcategory of Nori's category of effective cohomological isomotives generated by the first cohomology motive of pairs. Both categories are equivalent by \cite[Theorem 5.2, Theorem 6.1]{MR3302623}.} 
is the universal category associated with the representation
\begin{align}\label{eq:betti-hom}
H_1^{\Betti} : \Crv_k & \ra \mod(\bbQ)\\
(C,Y) & \mapsto  H^{\Betti}_1(C,Y)\otimes_\bbZ \bbQ
\end{align}
where $H^{\Betti}_1(C,Y)$ is the Betti homology of the pair $(C,Y)$
(with integral coefficients).
Let us denote by $\ECM_1$ the universal category associated with the representation
\begin{align*}
H^1_{\Betti} :\Crv_k^\op & \ra \mod(\bbQ)\\
(C,Y) & \mapsto  H^1_\Betti(C,Y)\otimes_\bbZ \bbQ
\end{align*}
where  $H_{\Betti}^1(C,Y)$ is the Betti cohomology of the pair $(C,Y)$. 
The $\bbQ$-linear dual functor $\mod(\bbQ)^\op \to \mod(\bbQ)$
induces an equivalence
\begin{equation}\label{eq:anti-eq-ehm-ecm}
(\EHM_1)^\op \to \ECM_1.
\end{equation}

\subsection{}\label{sect:deRhamMotives} In this work, it will be convenient to define effective cohomological motives of curves using algebraic de Rham cohomology instead of Betti cohomology. For this we assume that $k$ is a number field and consider the representation
\begin{align}
H^1_\dR:\Crv^\op & \ra \mod(\bbQ) \label{dRRep}\\
(C,Y) & \mapsto  H^1_\dR(C,Y):=\bH_\dR^1(\overline{C},Y,C_\infty) \notag
\end{align}
where $\overline{C}$ is the smooth compactification of $C$, 
$C_\infty=\overline{C}\setminus C$ is the set of points at infinity 
and $ \bH_\dR^1(\overline{C},Y,C_\infty)$ is defined as in \definitionref{DefidR} with both $Y$, $C_\infty$ viewed as closed reduced subschemes of $\overline{C}$.
Let us denote by $\ECM_1^\dR$ the $\bbQ$-linear Abelian category associated with the representation $H^1_\dR$ in \eqref{dRRep}.
By construction the representation $H^1_\dR$ has a factorization 
$$\Crv^\op\xra{\overline{H}^1_\dR} \ECM_1^\dR \xra{F_{\dR}}\mod(\bbQ) $$
into a representation $\overline{H}^1_\dR$ and a $\bbQ$-linear faithful exact functor $F_{\dR}$. Note that, by the universal property, the functor $F_\dR$ factorizes in $\mod(k)$ via the forgetful functor

\begin{lemm}\label{IsoPer}
There is a canonical isomorphism of functors 
$H^1_\dR\otimes_k \bbC\xra{\sim} H^1_{\Betti, \bbC}$ on  the category $\Crv$.
\end{lemm}
\begin{proof}
%
For a $k$-variety $V$ we write $V^\an$ 
for the complex analytic variety associated with $V$.
Let $(C,Y)$ in $\Crv$ and let $\Isheaf$, $\Jsheaf$ be the ideals of $Y^\an$ and $C_\infty^\an$ in $\Osheaf_{\overline{C}^{\raisebox{-2pt}{{\tiny an}}}}$.
The canonical map
\begin{equation}\label{IsoA}
\bH_\dR^1(\overline{C},Y,C_\infty)\otimes_k \bbC\ra  \bH^1(\overline{C}^{\raisebox{-2pt}{{\tiny an}}},[\Isheaf\ra\Jsheaf^{-1}\Omega^1_{\overline{C}^{\raisebox{-2pt}{{\tiny an}}}}])
\end{equation}
is an isomorphism of $\bbC$-vector spaces by GAGA. 
On the other hand, 
we have canonical quasi-isomorphisms 
\[
j_*\bbC_{C^\an} \cong
[\Osheaf_{\overline{C}^{\raisebox{-2pt}{{\tiny an}}}}
\ra\Jsheaf^{-1}\Omega^1_{\overline{C}^{\raisebox{-2pt}{{\tiny an}}}}],
\qquad
i_* \bbC_{Y^\an} \cong
[\Osheaf_{\overline{C}^{\raisebox{-2pt}{{\tiny an}}}}/\Isheaf\ra 0],
\]
where 
$j : C^\an \to \overline{C}^{\raisebox{-2pt}{{\tiny an}}}$ 
and 
$i : Y^\an \to \overline{C}^{\raisebox{-2pt}{{\tiny an}}}$ 
are immersions and
$\bbC_{C^\an}$ 
(resp. $\bbC_{Y^\an}$) denotes the constant sheaf on 
$C^\an$ (resp. $Y^\an$).
There is an exact sequence of complexes
\[
0 \to 
[\Isheaf\ra\Jsheaf^{-1}\Omega^1_{\overline{C}^{\raisebox{-2pt}{{\tiny an}}}}]
\to
[\Osheaf_{\overline{C}^{\raisebox{-2pt}{{\tiny an}}}}\ra \Jsheaf^{-1}\Omega^1_{\overline{C}^{\raisebox{-2pt}{{\tiny an}}}}]
\to
[\Osheaf_{\overline{C}^{\raisebox{-2pt}{{\tiny an}}}}/\Isheaf\ra 0]
\to 0.
\]
Hence the lemma follows from 
the fact that
$H^i_\Betti(C,Y)\otimes_\bbZ\bbC$
is computed as the hypercohomology of
the cone of $j_* \bbC_{C^\an} \to i_* \bbC_{Y^\an}$ with degree shifted by one.
\end{proof}

\begin{prop}\label{prop:equiv-betti-dr}
Let $k$ be a number field. The categories $\ECM_1$ and $\ECM_1^\dR$ are equivalent.
\end{prop}

\begin{proof}
Consider the 2-fiber product $\sA$ of the categories $\mod(k)$ and $\mod(\bbQ)$ over $\mod(\bbC)$. An object of $\sA$ is thus a triplet $(V,W,\alpha)$ where $V$ is a finite dimensional $k$-vector space, $W$ is a finite dimensional $\bbQ$-vector space and $\alpha:V\otimes_k \bbC\ra W\otimes_\bbQ \bbC$ is an isomorphism of $\bbC$-vector spaces. The category $\scr A$ is a $\bbQ$-linear Abelian category with two $\bbQ$-linear exact faithful functors
$$\Pi_1:\scr A\ra \mod(\bbQ)\quad \Pi_2:\scr A\ra\mod(\bbQ)$$
given by the projection on the first factor composed with the forgetful functor and the projection on the second factor.
We may then consider the representation
\begin{align*}
H^1_{\dR,\Betti}:\Crv^\op & \ra \scr A \\
(C,Y) & \mapsto H^1_{\dR,\Betti}(C,Y):=(H^1_\dR(C,Y),H^1_\Betti(C,Y)\otimes_\bbZ \bbQ,\alpha)
\end{align*}
where the isomorphism $\alpha:H^1_\dR(C,Y)\otimes_k \bbC\ra H^1_\Betti(C,Y)\otimes_\bbZ\bbC$ is the one of  \lemmaref{IsoPer}. 
We have the commutative diagram
$$\xymatrix{{} & {\ECM_1^\dR}\ar[r]\ar[r]\ar@{.>}[d]^-{\overline{\Pi}_1}  & {\mod(\bbQ)}\\
{\Crv^\op}\ar[r]^{H^1_{\dR,\Betti}}\ar@/_2em/[rd]_-{\overline{H}^1_{\Betti}}\ar@/^2em/[ru]^-{\overline{H}^1_\dR} & {\scr A}\ar[rd]^-{\Pi_2}\ar[ru]_-{\Pi_1} & {}\\
{} & {\ECM_1}\ar[r]\ar@{.>}[u]_-{\overline{\Pi}_2} & {\mod(\bbQ)}} $$
where $\overline{\Pi}_1$ and $\overline{\Pi}_2$ are the functors provided by the universal properties. The subdiagram 
$$\xymatrix{{} & {\ECM_1^\dR}\ar[rd]^-{\Pi_2\circ\overline{\Pi}_1} & {}\\
{\Crv^\op}\ar[ru]^-{\overline{H}^1_{\dR}}\ar[rd]_-{\overline{H}^1_{\Betti}} & {} & {\mod(\bbQ)}\\
{} & {\ECM_1}\ar[ru]\ar@{.>}[uu] & {}} $$
provides then a $\bbQ$-linear functor $\ECM_1 \ra\ECM_1^\dR$. Similarly we get a $\bbQ$-linear functor $\ECM_1^\dR\ra\ECM_1$ and it is easy to check that they are quasi-inverse to one another.
\end{proof}

Let $C$ be any smooth affine $k$-curve. We denote by $\overline{C}$ its smooth compactification and set $C_\infty=\overline{C}\setminus C$ viewed as a reduced subscheme of $\overline{C}$. This induces a morphism of quivers
\begin{align}\label{eq:crv-to-omcrv}
\overline{(-)}:\Crv & \ra \oMCrv\\
(C,Y) &\mapsto (\overline{C},Y,C_\infty). \notag
\end{align}

\begin{rema}\label{RemaCorrection}
Let $f:(C,Y)\ra (C',Y')$ be a morphism in $\Crv$. Then, $f$ extends to a morphism $\overline{f}:\overline{C}\ra\overline{C}'$ between smooth compactifications. This morphism satisfies $f(\overline{C}\setminus C_\infty)\subset \overline{C}'\setminus C'_\infty$ and 
since $f(Y)\subset Y'$, we have
$$Y=Y_\red\leqslant (f^{*}(Y_\red'))_\red\leqslant f^*(Y_\red')=f^*(Y').$$
Therefore, $\overline{f}$ defines a morphism between $(\overline{C},Y,C_\infty)$ and $(\overline{C}',Y',C'_\infty)$ in $\oMCrv$.
Similarly, we have another morphism of quivers
\begin{align}\label{eq:crv-to-omcrv2}
\Crv & \ra \uMCrv\\
(C,Y) &\mapsto (\overline{C},C_\infty, Y). \notag
\end{align}
\end{rema}

Since by definition  $H^1_{\dR}=\bH^1_{\dR}\circ\overline{(-)}$ as representations of the quiver $\Crv^\op$, the universal property of Nori's construction (see e.g. \cite[Theorem 2]{IvorraPNM}) ensures the existence of a $\bbQ$-linear exact faithful functor 
$$I_{\ECM}:\ECM_1^\dR \ra\ECMM_1$$
and isomorphisms of functors
\[
I_{\ECM} \circ \overline{H}^1_\dR \to
\overline{\bH}^1_\dR \circ \overline{(-)},
\qquad
 F^a_\dR \circ I_{\ECM}\ra F_\dR
\]
that makes the square 
$$\xymatrix{{F^a_\dR \circ I_{\ECM} \circ \overline{H}^1_\dR}\ar[r]\ar[d] & {F^a_\dR \circ \overline{\bH}^1_\dR \circ \overline{(-)}}\ar[d]^-{=}\\
{F_\dR \circ \overline{H}^1_\dR}\ar[r]^-{=} & {H^1_\dR}} $$
commutative.

Let us consider now the $\bbQ$-linear Abelian category $\sB$ defined as follows. An object in $\scr B$ is a tuplet $(V,W,a,b)$ where $V,W$ are finite dimensional $k$-vector spaces and $a:V\ra W$ and $b:W\ra V$ are morphisms of $k$-vector spaces such that $b\circ a=\Id$. A morphism 
$$(V,W,a,b)\ra (V',W',a',b') $$
in $\scr B$ is simply a pair of $k$-linear morphisms $(f:V\ra V',g:W\ra W')$ such that $a'\circ f=g\circ a$ and $b'\circ g=f\circ b$. Note that by construction, we have two $\bbQ$-linear exact functors obtained by projection on the first and second factor composed with the forgetful functor.
$$\Pi_1:\scr B\ra \mod(\bbQ)\qquad \Pi_2:\scr B\ra \mod(\bbQ)$$
and that moreover $\Pi_2$ is faithful.

 Let $X$ be a smooth proper $k$-curve and $Y,Z$ be closed subschemes of $X$. Recall from \propositionref{CanDec} that there are two morphisms
\begin{equation}\label{Mora}
a:\bH^1_\dR(X,Y_\red,Z_\red)\ra \bH^1_\dR(X,Y,Z)
\end{equation}
and 
\begin{equation}\label{Morb}
b:\bH^1_\dR(X,Y,Z)\ra \bH^1_\dR(X,Y_\red,Z_\red)
\end{equation}
such that $b\circ a=\id$.
We may therefore consider the representation
\begin{align*}
\bH^1_{\dR,\scr B}:\oMCrv^\op&\ra \scr B\\
(X,Y,Z) & \mapsto (\bH^1_\dR(X,Y_\red,Z_\red),\bH^1_\dR(X,Y,Z), a, b)
\end{align*}
where $a$ and $b$ are the morphisms \eqref{Mora} and \eqref{Morb}. 
By construction $\Pi_2\circ \bH^1_{\dR,\scr B}=\bH^1_\dR$ 
and from \eqref{dRRep} 
we have $\Pi_1\circ \bH^1_{\dR,\scr B}=H^1_\dR\circ (-)_\et$ where $(-)_\et$
is the morphism of quivers 
\begin{align*}
(-)_\et : \oMCrv & \ra \Crv\\
(X,Y,Z)& \mapsto (X\setminus Z_\red, Y_\red).
\end{align*}

By \cite[Theorem 2]{IvorraPNM}, there exists a faithful exact $\bbQ$-linear functor $F^a_{\scr B}:\ECMM_1\ra\scr B$ and two isomorphisms of functors $\gamma:F^a_\dR\circ\overline{\bH}^1_{\dR}\ra \bH^1_{\dR,\scr B}$, $\delta:\Pi_2\circ F^a_\scr B\ra F^a_\dR $ such that 
$$\xymatrix{{\Pi_2\circ F^a_{\scr B}\circ \overline{\bH}^1_\dR}\ar[r]^-{\Pi_2\vc\gamma}\ar[d]^-{\delta\vc \overline{\bH}^1_\dR} & {\Pi_2\circ \bH^1_{\dR,\scr B}}\ar[d]^-{=}\\
{F^a_\dR\circ \overline{\bH}^1_\dR}\ar[r]^-{=} & {\bH^1_\dR}} $$
is commutative.

We may apply \cite[Proposition 6.6]{IvorraPNM} to $\Pi_1$ to obtain the existence of  a $\bbQ$-linear exact and faithful functor  
$$\Pi_{\ECM}:\ECMM_1 \ra \ECM_1^\dR$$
and isomorphisms of functors
\[ \Pi_1 \circ F_\sB^a \to F_\dR \circ \Pi_{\ECM}, \quad
\Pi_{\ECM} \circ \overline{\bH}^1_\dR \to \overline{H}^1_\dR \circ (-)_\et
\]
such that the diagram
\[
\xymatrix{
\Pi_1 \circ F_\sB^a \circ \overline{\bH}^1_\dR \ar[r] \ar[d]^-{\Pi_1\vc\gamma}
&
F_\dR \circ \Pi_{\ECM} \circ \overline{\bH}^1_\dR \ar[r]
&
F_\dR  \circ \overline{H}^1_\dR \circ (-)_\et \ar[ld]^-{=}
\\
\Pi_1 \circ \bH^1_{\dR,\scr B} \ar[r]_= 
&
H^1_{\dR} \circ (-)_\et
}
\]
commutes. (See \cite[Proposition 6.7]{IvorraPNM} for uniqueness.)

%
%
\begin{rema} Let $I_{\scr B}:\mod(k)\ra\scr B$ be the functor that maps $V$ to $(V,V,\Id,\Id)$.
The diagram 
$$\xymatrix{{\ECM_1^\dR}\ar[r]^-{F_\dR}\ar[d]^-{I_\ECM} & {\mod(k)}\ar[d]^-{I_\scr B}\ar[rd]\\
{\ECMM_1}\ar[r]^-{F^a_\scr B}\ar@/_2em/[rr]_-{F^a_\dR} & {\scr B}\ar[r]^-{\Pi_2} & {\mod(\bbQ)}} $$
is commutative up to isomorphisms of functors.
\end{rema}

\begin{prop}\label{EmbeddingMotives}
The composition $\Pi_{\ECM}\circ I_{\ECM}$ is isomorphic to the identity.  Moreover the functor $I_{\ECM}$ is fully faithful. \end{prop}

\begin{proof}
Since $(-)_\et\circ \overline{(-)}$ is the identity on the quiver $\Crv$ the first assertion is an immediate consequence of the uniqueness statement \cite[Proposition 6.7]{IvorraPNM}. Let $M,N$ be objects in $\ECM_1^\dR$ and $\alpha:I_{\ECM}(M)\ra I_{\ECM}(N)$ be a morphism in $\ECMM_1$. Note that for such an $\alpha$, we have 
$$\Pi_1\circ F^a_{\sB}(\alpha)=\Pi_2\circ F^a_{\sB}(\alpha)=F^a_{\dR}(\alpha).$$
Let $\beta=\Pi_{\ECM}(\alpha)$. It is enough to show that  $I_{\ECM}(\beta)=\alpha$
and since $F^a_{\dR}$ is faithful, it is enough to show this equality after applying $F^a_{\dR}$. We have
\begin{align*}
F^a_{\dR}(I_{\ECM}(\beta)) &= F_{\dR}(\beta) = F_{\dR}(\Pi_{\ECM}(\alpha))\\
& =\Pi_1\circ F^a_{\scr B}(\alpha)=\Pi_2\circ F^a_{\sB}(\alpha)=F^a_{\dR}(\alpha).
\end{align*}
This concludes the proof.
\end{proof}

\section{Review of Laumon $1$-motives and its de Rham realization}
In this section, we recall necessary material from \cite{Laumon}, \cite{BVB},
introducing notations.

\subsection{}
Recall that we are working over a field $k$ of characteristic zero.
Let $\Aff$ be the category of affine schemes over $k$,
and let $\sS$ be 
the category of sheaves of Abelian groups on the fppf site on $\Aff$.
For $F \in \sS$,
we abbreviate $F(R):=F(\Spec R)$ for a $k$-algebra $R$,
and we put $\Lie(F):=\ker[F(k[\epsilon]/(\epsilon^2)) \to F(k)]$.

\subsection{}
We shall consider full subcategories of $\sS$.

Let $\sS_0$ be the full subcategory of $\sS$
consisting of objects that are represented by
connected commutative algebraic groups $G$ over $k$
(see \cite[(4.1)]{Laumon}).
We identify such a $G$ with the object in $\sS$ represented by $G$.

Let $\sS_l$ be the full subcategory of $\sS_0$
consisting of linear commutative algebraic groups over $k$,
We write $\sS_\uni$ (resp. $\sS_\mul$) for the full subcategory
of $\sS_l$ consisting unipotent (resp. multiplicative) groups.
For any $L \in \sS_l$,
there is a canonical decomposition
$L \cong L_{\uni} \times L_{\mul}$,
where $L_{\uni} \in \sS_\uni$ and $L_{\mul} \in \sS_\mul$.
The functor $\sS_{\uni} \to \mod(k), ~L \mapsto L(k)$ is an equivalence,
by which we often identify them.

Let $\sS_a$ be the full subcategory of $\sS_0$
consisting of Abelian varieties.
Recall that any $G \in \sS_0$ 
canonically fits in
an extension $0 \to G_l \to G \to G_{\ab} \to 0$,
where $G_{\ab} \in \sS_a$ and $G_l \in \sS_l$.
We ease the notation by putting 
$G_{\uni}=(G_l)_{\uni}$ and $G_{\mul}=(G_l)_{\mul}$.
We call $G_{\sa}:=G/G_{\uni}$ the \emph{semi-Abelian part} of $G$.

Let $\sS_{-1}$ be the full subcategory of $\sS$
consisting of formal groups over $k$ without torsion
in the sense of \cite[(4.2)]{Laumon}.
We write  $\sS_{\inf}$ (resp. $\sS_\et$) for
the full subcategory of $\sS_{-1}$
consisting of connected (resp. \'etale) formal groups.
For any $F \in \sS_{-1}$,
there is a canonical decomposition
$F \cong F_{\inf} \times F_{\et}$,
where $F_{\inf} \in \sS_{\inf}$ and $F_{\et} \in \sS_\et$.
The functor $\Lie : \sS_{\inf} \to \mod(k)$ is an equivalence,
with a quasi-inverse $V \mapsto V \otimes_k \hat{\bbG}_a$,
where $\hat{\bbG}_a$ denotes the formal completion of $\bbG_a$.

\subsection{}
Following \cite[(5.1.1)]{Laumon},
define a \emph{Laumon $1$-motive} to be a complex $[F \to G]$ in $\sS$
such that $F \in \sS_{-1}$ (placed at degree $-1$) 
and $G \in \sS_0$ (placed at degree $0$).
We denote the category of Laumon $1$-motives over $k$
by $\sM^a_1$ 
(or by $\sM^a_1(k)$ if we with to stress the dependency on $k$).
There is an equivalence $(\sM^a_1)^\op \to \sM^a_1$,
called \emph{Cartier duality}.

\subsection{}\label{sect:deligne-1-motives}
A Laumon $1$-motive $[F \to G]$ is called a 
\emph{Deligne $1$-motive} if $F_{\inf}=0$ and $G_\uni=0$.
Denote by $\sM_{1}$ the full subcategory of $\sM^a_1$
consisting of Deligne $1$-motives.
Along with this, 
we denote by $\sM_1^\uni$ (resp. $\sM_1^{\inf}$)
the essential image of an obvious full faithful functor
\begin{align*}
&\sS_\uni \to \sM^a_1, \quad
U \mapsto U[0]:=[0 \to U],
\\
(\text{resp.}~
&\sS_{\inf} \to \sM^a_1, \quad
F \mapsto F[1]:=[F \to 0]).
\end{align*}

\subsection{}\label{sect:def-grM}
Let $M=[F \to G] \in \sM^a_1$.
We define a filtration on $M$ by
\[ \fil^0_\sM M=M \supset \fil^1_\sM M=[F_\et \to G]
\supset \fil^2_\sM M=[0 \to G_\uni]
\supset \fil^3_\sM M=0.
\]
We put $\gr^i_\sM M := \fil^i_\sM M/\fil^{i+1}_\sM M$
so that
\[ 
\gr^0_\sM M \cong  F_{\inf}[1],~
\gr^1_\sM M \cong [F_{\et} \to G_\sa] =:M_\Del,~
\gr^2_\sM M = \fil^2_\sM M = G_\uni[0].
\]
We have defined functors
\[
\gr_\sM^0 : \sM^a_1 \to \sM_1^{\inf},\quad
\gr_\sM^1 : \sM^a_1 \to \sM_{1, \Del},\quad
\gr_\sM^2 : \sM^a_1 \to \sM_1^\uni.
\]
Note that all these functors are exact,
and that
$\gr_\sM^0$ (resp. $\gr_\sM^2$)
is a left (resp. right) adjoint to the inclusion
$\sM_1^{\inf} \hookrightarrow \sM^a_1$
(resp. $\sM_1^{\uni} \hookrightarrow \sM^a_1$).
Following \cite{BVB}, we also define (recall that $G_\sa=G/G_\uni$)
\[ M_\times := M/\fil^2_\sM M = [F \to G_\sa]. \]
The functor $M \mapsto M_\times$ is a left adjoint of the inclusion $\{ G \in \sM^a_1 ~|~ G_\uni=0 \} \hookrightarrow \sM^a_1$.

\subsection{}
We call $M=[F \to G] \in \sM^a_1$
\emph{unipotent free} if $G_\uni=0$.
For such $M$, 
it is shown in \cite[(2.2.3)]{BVB}
that there is
an extension
$M^{\natural} = [F \to G^\natural] \in \sM^a_1$
of $M$ by $\Ext_{\sM_1^a}(M, \bbG_a)^*$
such that it is universal among extensions of $M$
by an object of $\sM_1^{\uni}$.
(Here 
${}^*$ denotes $k$-linear dual.
Recall that by convention we identify a $k$-vector space with
an object of $\sS_\uni$.)

\subsection{}
Now take any $M=[u : F \to G] \in \sM^a_1$.
Note that  $M_\times$ and $M_\Del$ 
(introduced in \pararef{sect:def-grM})
are unipotent free.
By \cite[(2.3.2)]{BVB}, an exact sequence
\[ 0 \to M_\Del \to M_\times \to F_{\inf}[1] \to 0 \]
induces an exact sequence
\[ 0 \to (M_\Del)^\natural \to (M_\times)^\natural 
\to \vec{F}_{\inf} \to 0,
\]
where
$\vec{F}_{\inf} 
:= [F_{\inf} \to \Lie(F_{\inf})] \in \sM^a_1$.
Let us write 
$(M_\Del)^\natural=[u_\Del^\natural : F_\et \to G_\Del^\natural]$
and 
$(M_\times)^\natural=[u_\times^\natural : F \to G_\times^\natural]$.
Then we get an exact sequence
\begin{equation}\label{eq:can-spl1}
 0 \to \Lie(G_\Del^\natural) \to 
\Lie(G_\times^\natural) 
\to \Lie(F_{\inf}) \to 0,
\end{equation}
which admits a canonical splitting
given by $\Lie(u_\times^\natural)$.

We also need the following remark.
The universality of $(M_\times)^\natural$
induces maps $v_M$ and $v_M^\natural$ in
the following commutative diagram
with exact rows
\begin{equation}\label{eq:v_m}
\xymatrix{
0 \ar[r]
&
\Ext(M_\times, \bbG_a)^*
\ar[r] \ar[d]^{v_M}
& 
(M_\times)^\natural
\ar[r] \ar[d]^{v_M^\natural}
& 
M_\times
\ar[r] \ar[d]^{=}
& 0
\\
0 \ar[r]
&
G_\uni
\ar[r]
& 
M
\ar[r]
& 
M_\times
\ar[r]
& 0.
}
\end{equation}

\subsection{}
The \emph{sharp extension} $M^\sharp=[F \to G^\sharp]$ of 
$M=[F \to G] \in \sM_1^a$
is defined to be the pull-back of
$(M_\times)^\natural$ by the canonical surjection $M \to M_\times$.
(If $M$ is unipotent free, then $M^\sharp = M^\natural$.)
There is a commutative diagram with
exact rows and coloums
\begin{equation}\label{eq:big-diag}
\xymatrix{
&
& 
0 \ar[d]
& 
0 \ar[d]
&
\\
&
& 
G_\uni
\ar[r]^{=} \ar[d]
& 
G_\uni  \ar[d]_{i}
&
\\
0 \ar[r]
&
\Ext(M_\times, \bbG_a)^*
\ar[r] \ar[d]_{=}
& 
M^\sharp
\ar[r]_{p} \ar[d]_q
& 
M
\ar[r] \ar[d]
& 0
\\
0 \ar[r]
&
\Ext(M_\times, \bbG_a)^*
\ar[r]
& 
(M_\times)^\natural
\ar[r] \ar@{-->}[ur]_{v_M^\natural}  \ar[d]
& 
M_\times
\ar[r]  \ar[d]
& 0
\\
&
& 
0
& 
0.
&
}
\end{equation}
Note that the dotted arrow $v_M^\natural$ makes
the lower right triangle commutative by \eqref{eq:v_m},
but it is \emph{not} necessarily the case for the upper left triangle.
The middle vertical exact sequence in \eqref{eq:big-diag} admits 
a {canonical} splitting $s : M^\sharp \to G_\uni$
characterized by $i \circ s = p - (v_M^\natural \circ q)$.
Hence there also is an exact sequence
\begin{equation}\label{eq:can-spl2}
0
\to 
\Lie(G_\uni)
\to \Lie(G^\sharp)
\to \Lie(G_\times^\natural)
\to 0
\end{equation}
equipped with a {canonical} splitting.
Combined with \eqref{eq:can-spl1}, we obtain
a {canonical} decomposition
\begin{equation}\label{eq:can-spl3}
\Lie(G^\sharp)
\cong
\Lie(G_\Del^\sharp) \oplus \Lie(F_{\inf}) \oplus \Lie(G_\uni).
\end{equation}
%
%
%
%

\subsection{}
Following \cite[(3.2.1)]{BVB}, we call
an exact functor 
\[ \drR_\dR : \sM^a_1 \to \mod(k),
\quad \drR_\dR([F \to G]) := \Lie(G^\sharp).
\]
the \emph{sharp de Rham realization}.
By \eqref{eq:can-spl3},
we have a {canonical} decomposition
\begin{equation}\label{eq:can-spl4}
\drR_\dR(M) \cong
\drR_\dR(M_{\Del})\oplus \Lie(F_{\inf}) \oplus \Lie(G_\uni)
\end{equation}
for any $M=[F \to G] \in \sM^a_1$.

\subsection{}
Let $\sM_{1,\bbQ}:=\sM_1 \otimes_{\bbZ} \bbQ$ be
the $\bbQ$-linear Abelian category of Deligne $1$-isomotives
(see \eqref{eq:scalar-ext}).
Recall from \pararef{sect:EHM-ECM} that
$\EHM_1^{\bbQ}$ is the universal $\bbQ$-linear category
associated with the Betti homology functor
\eqref{eq:betti-hom} (with $K=\bbQ$).
In \cite{MR3302623}, \name{L. Barbieri-Viale} and \name{J. Ayoub} show 
the following important result\footnote{Actually, they prove a stronger statement with integral coefficients.}
which will be a key ingredient in the proof of our main result.

\begin{theo}[Ayoub \& Barbieri-Viale \cite{MR3302623}]\label{EquivDel}
We have an equivalence of 
$\bbQ$-linear Abelian categories
$$\EHM_{1}^\bbQ \xra{\sim}\sM_{1,\bbQ}.$$
\end{theo}

This functor is induced 
by a functor $\Crv \to \sM_1$ via universality
(see \remarkref{rem:a-bv-functor}).
We will construct its modulus version in the next section.

\section{1-motives of a curve with modulus and the main theorem}
In this section, 
to a smooth proper $k$-curve $X$ and
two effective divisors $Y, Z$ on $X$ with disjoint support,
we associate 
a Laumon $1$-motive
$\LM(X, Y, Z) \in \sM_1^a$.
We shall see functorial properties
that yield two functors
$\oLM : \oMCrv \to \sM^a_1$ and $\uLM : \uMCrv \to \sM^a_1$.

\subsection{}\label{sect:gen-jac}
Let $X$ be a smooth proper $k$-curve and $Y$ an effective divisors on $X$.
We denote by $J(X, Y) \in \sS_0$ the \emph{generalized Jacobian}
of $X$ with modulus $Y$ in the sense of Rosenlicht-Serre \cite{Rosenlicht,SerreGACC}.
Recall that
$J(X, Y)$ is the connected component of the Picard scheme
$\underline{\Pic}(X_Y)$ of a proper $k$-curve $X_Y$
that is obtained by collapsing $Y$ into a single (usually singular) point
(see \cite[Chapter IV, \S 3--4]{SerreGACC}).
It can also be defined 
as the Albanese variety attached to a pair $(X,Y)$ \cite[Example 2.34]{MR2513591}, \cite[\S 3.3]{MR3095229}.

Let $X'$ be another smooth proper $k$-curve
and $Y'$ an effective divisor on it.
Let $f: X \to X'$ be a $k$-morphism.
When $Y \le f^*Y'$, we have a pull-back
$f^* : J(X', Y') \to J(X, Y)$
deduced by the functoriality of the Picard scheme.
When $Y-Y_\red \ge f^*(Y'-Y_\red')$
and $Y_\red \ge (f^* Y')_\red$, we have a push-forward
$f_* : J(X, Y) \to J(X', Y')$
by \cite[Proposition 3.22]{MR3095229}.

\begin{lemm}
There exists a canonical isomorphism
(cf. \propositionref{prop:uv})
\begin{equation}\label{eq:lie-jacuni}
\Lie J(X, Y)_\uni \cong U(X, Y).
\end{equation}
\end{lemm}
\begin{proof}
If $Y=\emptyset$, then $J(X, Y)$ is an Abelian variety
so that $J(X, Y)_\uni=0$, and hence the lemma holds.
We suppose $Y \not= \emptyset$ in what follows.
Consider an exact sequence of sheaves on $X$:
\[ 0 \to \Isheaf_Y \to \Isheaf_{Y_{\red}} \to 
\Isheaf_{Y_\red}/\Isheaf_Y \to 0.
\]
We have $H^0(X, \Isheaf_{Y_{\red}})=0$
since $Y$ is a non-empty effective divisor.
It follows that
\[
H^0(X, \Isheaf_{Y_\red}/\Isheaf_Y) \cong
\ker(H^1(X, \Isheaf_Y) \to H^1(X, \Isheaf_{Y_{\red}})).
\]
By \cite[Chapter V, \S 10, Proposition 5]{SerreGACC},
there are canonical isomorphisms
\[
H^1(X, \Isheaf_Y) \cong \Lie J(X, Y), \quad
H^1(X, \Isheaf_{Y_{\red}}) \cong \Lie J(X, Y_\red).
\]
Now the lemma follows from an exact sequence
\[
0 \to \Lie J(X, Y)_{\uni}
\to \Lie J(X, Y) \to \Lie J(X, Y)_\sa \to 0
\]
and a canonical isomorphism $J(X, Y)_{\sa}=J(X, Y_\red)$.
\end{proof}

\subsection{}
Let $X$ be a smooth proper $k$-curve and $Z$ an effective divisors on $X$.
We construct an object
\begin{equation}
F(X, Z):= F(X, Z)_{\inf} \times F(X, Z)_{\et} \in \sS_{-1}
\end{equation}
as follows. First, we define 
\[ F(X, Z)_{\et}:=\ker[\pi_0(Z) \to \pi_0(X)], \]
where the map is the one induced by the closed immersion $Z \to X$.
Here for any $k$-variety $V$, 
we define $\pi_0(V) \in \sS_{-1}$ by
declaring $\pi_0(V)(U)$ 
is the free Abelian group on the set of connected components of
$U \times_k V$ for $U \in \Aff$.
This depends only on the reduced part of $V$.
Next, we define (cf. \propositionref{prop:uv}, see also \cite[\S 5.3]{MR2985516})
\begin{equation}\label{eq:def-f-inf}
 F(X, Z)_{\inf}:=V(X, Z) \otimes_k \hat{\bbG}_a.
\end{equation}

Let $X'$ be another smooth proper $k$-curve
and $Z'$ an effective divisor on it.
Let $f: X \to X'$ be a $k$-morphism.
There is  
a pull-back $f^* : F(X', Z') \to F(X, Z)$
(resp. a push-forward $f_* : F(X, Z) \to F(X', Z')$)
when $Z-Z_\red \ge f^*(Z'-Z_\red')$ and $Z_\red \ge (f^* Z')_\red$
(resp. $Z \le f^*Z'$).
On the infinitesimal (resp. \'etale) part, 
they are defined by \corollaryref{prop:pull-push-uv}
(resp. pull-back and push-forward of cycles).

\subsection{}
We recall from \cite[\S 2.1]{MR2513591} Russell's results that we will use.
Let $V$ be a noetherian reduced scheme.
Define $\underline{\Div}_V \in \sS$ to be
the sheaf that associates to $\Spec(R) \in \Aff$
the group of all Cartier divisors on $V \otimes_k R$
generated locally on $\Spec(R)$ by 
effective Cartier divisors which are flat over $R$.
There is a canonical ``class'' map 
\begin{equation}\label{eq:cl-map}
\cl : \underline{\Div}_V \to \underline{\Pic}_V
\end{equation}
to the Picard scheme $\underline{\Pic}_V$ of $V$.
Denote by $\underline{\Div}^0_V$
the inverse image of the connected component
$\underline{\Pic}_X^0$ of $\underline{\Pic}_X$ by $\cl$.
We have $\underline{\Div}_V^0(k)=H^0(V, \Ksheaf_V^\times/\Osheaf_V^\times)$
(the group of Cartier divisors on $V$)
and $\Lie(\underline{\Div}_V^0)=H^0(V, \Ksheaf_V/\Osheaf_V)$,
where $\Ksheaf_V$ is the sheaf of total ring of fractions of $\Osheaf_V$.
In \cite[Proposition 2.13]{MR2513591} it is shown that
for any $F \in \sS_{-1}$ and
a pair of maps 
\[
a_{\inf} : \Lie(F) \to \Lie(\underline{\Div}_V^0),
\quad
a_{\et} : F(k) \to \underline{\Div}_V^0(k),
\]
there exists a unique map 
\begin{equation}\label{eq:map-to-div}
a=(a_{\inf}, a_{\et}) : F \to \underline{\Div}_V
\end{equation}
that induces a map $a_{\inf}$ (resp. $a_{\et}$)
via $\Lie$ 
(resp. by taking sections over $\Spec k$).

Let $X$ be a smooth proper $k$-curve and 
let $Y, Z$ be two effective divisors on $X$ with disjoint support.
We apply the above argument to $V=X_Y$, 
where $X_Y$ is the curve we discussed in \pararef{sect:gen-jac}.
Since $Y$ and $Z$ are disjoint,
we may identify $Z$ as a closed subscheme of $X_Y$.
We define
\begin{align*}
\tau_{\inf}' : 
&\Lie(F(X, Z)_{\inf})
= H^0(X, \Isheaf_Z^{-1}\Isheaf_{Z_\red}/\Osheaf_X)
= H^0(X_Z, \Isheaf_Z^{-1}\Isheaf_{Z_\red}/\Osheaf_{X_Y})
\\
&\to H^0(X, \Ksheaf_{X_Y}/\Osheaf_{X_Y})
= \Lie(\underline{\Div}_{X_Y}^0)
\end{align*}
to be the map induced by the inclusion
$\Isheaf_Z^{-1}\Isheaf_{Z_\red} \subset \Ksheaf_{X_Y}$.
Also, we define
\[ \pi_0(Z)(k)=Z_0(Z) 
\to \underline{\Div}_{X_Y}(k) = \Div(X_Y) 
\]
by sending $D \in Z_0(Z)$ to $\Osheaf_{X_Y}(D)$.
It restricts to 
\[
\tau_{\et}' : 
F(X, Z)_\et(k) = \ker[\pi_0(Z) \to \pi_0(X)]
\to \underline{\Div}_{X_Y}^0(k).
\]
Using them, we define 
\[ \tau(X, Y, Z):=\cl \circ (\tau_{\inf}', ~\tau_{\et}') : 
F(X, Z) \to \underline{\Pic}_{X_Y}^0 = J(X, Y),
\]
where we used the notations from
\eqref{eq:cl-map} and \eqref{eq:map-to-div}.
We then define a Laumon $1$-motive 
attached to $(X, Y, Z)$ by
\begin{equation}\label{eq:def-lm}
\LM(X, Y, Z) 
:= [ F(X, Z) \overset{\tau(X, Y, Z)}{\longrightarrow} J(X, Y)]
\in \sM^a_1.
\end{equation}
From this definition it is evident that
\begin{equation}\label{eq:red-Del}
 \LM(X, Y, Z)_\Del = \LM(X, Y_\red, Z_\red).
\end{equation}

\subsection{}\label{FoncLM}
Let $X'$ be another smooth proper $k$-curve and 
let $Y', Z'$ be two effective divisors on $X'$ with disjoint support.
Let $f : X \to X'$ be a $k$-morphism.
If $f$ defines a morphism in $\oMCrv$,
then the square
$$\xymatrix@C=2cm{{F(X',Z')}\ar[r]^-{\tau(X',Y',Z')}\ar[d]^-{f^*} & {J(X',Y')}\ar[d]^-{f^*}\\
{F(X,Z)}\ar[r]^-{\tau(X,Y,Z)} & {J(X,Y)}} $$
commutes. 
Similarly 
if $f$ defines a morphism in $\uMCrv$, then the square
$$\xymatrix@C=2cm{{F(X,Z)}\ar[r]^-{\tau(X,Y,Z)}\ar[d]^-{f_*} & {J(X,Y)}\ar[d]^-{f_*}\\
{F(X',Z')}\ar[r]^-{\tau(X',Y',Z')} & {J(X',Y')}} $$
commutes. 
This enables us to make the following definition.

\begin{defi}
We define a  functor
\[ \oLM : \oMCrv^{\op} \to \sM^a_1, 
\quad (\text{resp. } \uLM : \uMCrv \to \sM^a_1)
\]
by setting
\begin{equation}
\oLM(X, Y, Z) 
=\uLM(X, Y, Z) 
=\LM(X, Y, Z),
\end{equation}
and $\oLM(f)=f^*$ (resp. $\uLM(f)=f_*$)
for a morphism
$f$ in $\oMCrv$ (resp. in $\uMCrv$).
\end{defi}

\begin{rema}\label{rem:a-bv-functor}
The composition of $\uLM$ with 
$\Crv \to \uMCrv$ from \eqref{eq:crv-to-omcrv2}
factors through $\sM_1$ (see \pararef{sect:deligne-1-motives}).
This induces the functor in \theoremref{EquivDel}
via universality.
\end{rema}


\begin{prop}\label{prop:compati}
There is an isomorphism of functors 
$\drR_\dR \circ \oLM \to \bH^1_\dR$.
\end{prop}
\begin{proof}
Let $(X, Y, Z) \in \oMCrv$.
By \eqref{eq:can-spl4}, \eqref{eq:lie-jacuni}, 
\eqref{eq:def-f-inf} and \eqref{eq:red-Del},
we have
\begin{align*}
 \drR_\dR \circ \oLM(X, Y, Z) 
&= \drR_\dR(X, Y_\red, Z_\red) \oplus U(X, Y) \oplus V(X, Z).
\end{align*}
Moreover, by
\cite[Corollary 2.6.4]{MR1891270}
there is a canonical isomorphism
$\drR_\dR(X, Y_\red, Z_\red) \cong \bH^1_\dR(X, Y_\red, Z_\red)$.
Now the proposition follows from \eqref{eq:hdr-dec}.
\end{proof}

\begin{rema}
There is also an isomorphism of functors 
$\drR_\dR \circ \uLM \to {}^t\bH^1_\dR$,
considered as functors $\uMCrv \to \mod(k)$,
see \eqref{eq:functors-bh2}.
\end{rema}

\begin{rema}
(This remark will not be used in the sequel.)
For any $(X, Y, Z) \in \oMCrv$, we find that
$\LM(X, Y, Z)$ and $\LM(X, Z, Y)$ are Cartier dual to each other.
In other words, 
using a functor 
$\Sw : \uMCrv \to \oMCrv$ defined by $\Sw(X, Y, Z)=(X, Z, Y)$,
we get a commutative diagram
\[
\xymatrix{
\oMCrv 
\ar[d]_{\Sw}
\ar[r]^{\oLM^\op}
&
(\sM_{1}^a)^{\op}
\ar[d]^{\text{Cartier dual}}
\\
\uMCrv 
\ar[r]_{\uLM}
&
\sM_{1}^a.
}
\]
%
%
\end{rema}

\subsection{}
Let $\sM^a_{1, \bbQ} := \sM_1^a \otimes_\bbZ \bbQ$
be the $\bbQ$-linear Abelian category of Laumon $1$-isomotives
(see \eqref{eq:scalar-ext}).
\begin{prop}\label{QuotProp}
Any Laumon $1$-motive $M=[F \overset{u}{\to} G]$
is a quotient in $\sM^a_{1, \bbQ}$
of $\oLM(X, Y, Z)$ for some object $(X, Y, Z)$ of $\oMCrv$.
\end{prop}


\begin{rema}
If $M$ is such that $F=0$, then $(X, Y, Z)$ can be choosen as $Z=\emptyset$
Similarly,
if $M$ is such that $G_l=0$, then $(X, Y, Z)$ can be choosen as $Y=\emptyset$.
This will be apparent from the proof given below.
\end{rema}

\begin{proof}
We divide the proof into three steps.

{\it Step 1.} 
(Compare \cite[Chap. VII, \S 2, no. 13, Thm. 4]{SerreGACC}.)
We first prove the proposition 
assuming that $k$ is algebraically closed, 
and that both 
$u_{\inf} : \Lie(F_{\inf}) \to \Lie(G)$ and $u_{\et} : F_{\et} \to G$ 
are injective.
Choose a $\bbZ$-basis $e_1, \dots, e_r$ of $F_{\et}$,
and put $p_i:=u_{\et}(e_i) \in G ~(i=1, \dots, r)$. 
Let $p_0 \in G$ be the identity element.
We take a one-dimensional closed integral subscheme $C_0'$ on $G$ 
that contains $p_0, p_1, \dots, p_r$ as regular points.
Also, choose a $k$-basis $t_1, \dots, t_{s'}$ of $\Lie(F_{\inf})$,
and put $v_i:=u_{\inf}(t_i) \in \Lie(G) ~(i=1, \dots, s')$. 
We extend $v_1, \dots, v_{s'}$ to a $k$-basis 
$v_1, \dots, v_{s'}, \dots, v_s$ of $\Lie(G)$.
For each $i=1, \dots, s$,
we take a one-dimensional closed integral subscheme $C_i'$ on $G$ 
that passes $p_0$ regularly and that has tangent $v_i$ at $p_0$.
For $i=0, 1, \dots, s$, we let $C_i \to C_i'$ be the normalization.
We denote the preimage of $p_j$ in $C_i$ by the same letter $p_j$.
(Here $j=0, \dots, r$ for $i=0$, and $j=0$ for $i=1, \dots, s$.)
Let $X_i$ be the smooth completion of $C_i$.
Let $Y_i$ be a modulus for the morphism $C_i \to C_i' \hookrightarrow G$.
This means that $Y_i$ is an effective divisor supported on $X_i \setminus C_i$
and that $C_i \to G$ factors as 
$C_i \to J(X_i, Y_i) \overset{g_i}{\to} G$.
We also define effective divisors
$Z_0 := (p_0)+(p_1)+\dots+(p_r) \in \Div(X_0)$,
$Z_i :=2(p_0)  \in \Div(X_i) ~(i=1, \dots, s')$,
and $Z_i := 0 ~(i=s'+1, \dots, s)$.
Let $X$ be the disjoint union of $X_0, \dots, X_s$,
and let $Y=Y_0+\dots+Y_s, Z=Z_0+\dots+Z_s$.

By definition, we have
$F(X,Z)_{\et}=F(X_0, Z_0)=\Div_{Z_0}^0(X_0)$,
hence we can define an isomorphism
$F(X, Z)_{\et} \to F_{\et}$ by 
$\sum_{i=1}^r n_i (p_i-p_0) \mapsto \sum_{i=1}^r n_i e_i~(n_1, \dots, n_r \in \bbZ)$.
Also, by definition, we have
$F(X,Z)_{\inf}
=\oplus_{i=1}^{s'} F(X_i, Z_i)=\oplus_{i=1}^{s'} k \cdot v_i$,
hence we can define an isomorphism
$F(X, Z)_{\inf} \to F_{\inf}$ by 
$\sum_{i=1}^{s'} a_i v_i \mapsto \sum_{i=1}^{s'} a_i t_i ~(a_1, \dots, a_{s'} \in k)$.
We have defined an isomorphism $f : F(X, Z) \to F$.
Finally, we define $g : J(X, Y) \to G$ as the sum of 
$g_i : J(X_i, Y_i) \to G$ over $i=0, \dots, s$.
Since the image of 
$\Lie(g_i) : \Lie(J(X_i, Y_i)) \to \Lie(G)$ contains $v_i$,
we find $\Lie(g) : \Lie(J(X, Y)) \to \Lie(G)$ is surjective,
hence $g: J(X, Y) \to G$ itself is also surjective.
It is straightforward to see that $f$ and $g$ define
an epimorphism $\oLM(X, Y, Z) \to M$ in $\sM^a_1$.
(Here we do not need to tensor with $\bbQ$.)

{\it Step 2.} 
We drop the assumption that $k$ is algebraically closed,
but keep the assumption that both $u_{\inf}$ and $u_{\et}$ are injective.
By Step 1, we can find a finite extension $k'/k$
such that the base change of $M$ to $k'$ satisfies the conclusion of the proposition.
The Weil restriction functor 
\[ R_{k'/k} :\sM^a_{1, \bbQ}(k') \to \sM^a_{1, \bbQ}(k), \qquad
~R_{k'/k}([F \to G]) = [R_{k'/k}(F) \to R_{k'/k}(G)]
\]
is exact.
(Here we denote by $\sM^a_{1, \bbQ}(k)$ and $\sM^a_{1, \bbQ}(k')$ 
for the category of Laumon $1$-isomotives over $k$ and over $k'$.)
Moreover, for any $(X, Y, Z) \in \oMCrv_{k'}$
we have $R_{k'/k} \oLM_{k'}(X, Y, Z)=\oLM_{k}(X_k, Y_k, Z_k)$,
where for an $k'$-scheme $S$
we write $S_k$ for the $k$-scheme $S$ with structure morphism
$S \to \Spec k' \to \Spec k$.
(This follows from a general fact that
the Picard functor commutes with base change.)
This proves the proposition in this case.

{\it Step 3.} We prove the proposition in general case.
Let $F_2 := \ker(u),~ M_1:=[F/F_2 \to G], ~M_2:=[F_2 \to 0]$.
Then there is a non-canonical isomorphism 
$M \cong M_1 \oplus M_2$ in $\sM^a_{1, \bbQ}$.
Now we apply the result from Step 2, and we are done.
\end{proof}

\subsection{}\label{sect:start-assuming-k-is-Q}
From now on until the end of the paper,
we suppose that $k$ is a number field.
Note that $\sM_{1,\bbQ}^a$ is a $\bbQ$-linear Abelian category. 
By \propositionref{NoriLength} and  \propositionref{prop:compati},
we obtain a $\bbQ$-linear exact faithful functor 
\begin{equation}\label{eq:phi}
\bLM : \ECMM_1 \to \sM_{1,\bbQ}^a.
\end{equation}
and two invertible natural transformations $\bLM\circ\overline{\bH}^1_\dR\ra\oLM$, $\mathsf{R}_{\dR}\circ\bLM\ra F^a_\dR$. 
The main result of this article is the following:

\begin{theo}\label{MainTheo2}
Suppose that $k$ is a number field.
The functor
$$\bLM : \ECMM_1 \to \sM_{1,\bbQ}^a $$
in 
\eqref{eq:phi} is an equivalence. 
\end{theo}

\section{Filtration on Nori motives with modulus}\label{sect:filtration}
We keep assuming that $k$ is a number field.
In this section, we construct on every object of $\ECMM_{1}$ a two steps filtration that mirrors the one on   Laumon $1$-motives defined in \pararef{sect:def-grM}.

\subsection{}\label{Fil1} 
Consider the morphism of quivers
\begin{align}
\oMCrv  & \ra \oMCrv \label{MorMCrvD}\\
(X,Y,Z) & \mapsto (X,Y,Z_\red).\notag
\end{align}
Note that if a morphism $f : X \to X'$ of $k$-curves
defines 
a morphism $(X, Y, Z) \to (X', Y', Z')$ in $\oMCrv$,
then it also defines 
a morphism $(X, Y, Z_\red) \to (X', Y', Z_\red')$ in $\oMCrv$,
by our definition of $\oMCrv$ (see \S \ref{sect:def-MCrv}).



If  $(X,Y,Z)$ is a $k$-curve with modulus, let us observe that by construction $\fil^1_\sM\LM(X,Y,Z)=\LM(X,Y,Z_\red)$. 
Hence the square
\begin{equation}\label{squareFil1}
\xymatrix{{\oMCrv^\op}\ar[r]^-{\LM}\ar[d]_-{(-,-,-_\red)} & {\sM^a_{1,\bbQ}}\ar[d]^{\fil_\sM^1}\\
{\oMCrv^\op}\ar[r]^-{\LM} & {\sM^a_{1,\bbQ}}}
\end{equation}
commutes and \propositionref{Fonc3} shows the existence of a $\bbQ$-linear exact functor $\fil^1:\ECMM_{1}\ra\ECMM_{1}$ and two invertible natural transformations
$$\rho:\fil^1_\sM\circ \bLM\ra \bLM\circ \fil^1\qquad \varrho: \fil^1\circ \overline{\bH}^1_\dR\ra \overline{\bH}^1_\dR\circ(-,-,-_\red) $$
such that the corresponding diagram as in \eqref{Dia1Fonc3} is commutative.

Let us now show that there exists a natural transformation $\fil^1\ra \Id$ which is a monomorphism for every object in $\ECMM_{1}$.
Let $(X,Y,Z)$ be a $k$-curve with modulus. Since $Z_\red\leqslant Z$, the identity of $X$ defines an edge $(X,Y,Z)\ra (X,Y,Z_\red)$ that provides a natural transformation
$$\iota:(-,-,-_\red)\ra \Id $$
of functors from $\oMCrv^\op$ with values in $\oMCrv^\op$. Note that this transformation induces the monomorphism $\fil^1\LM(X,Y,Z)\ra\LM(X,Y,Z)$ in $\sM^a_{1,\bbQ}$ and that the square 
$$\xymatrix{{\fil^1_\sM\circ \LM}\ar@{=}[r]\ar[d]_-{\iota_\sM\vc\LM} & {\LM\circ (-,-,-_\red)}\ar[d]^-{\LM\vc\iota}\\
{\LM}\ar@{=}[r] & {\LM}} $$
is commutative. We may therefore apply \propositionref{NatTransf} to obtain a natural tranformation $\overline{\iota}:\fil^1\ra\Id $ that makes the squares
\begin{equation}\label{CompFil1}
\xymatrix{{\fil^1\circ \overline{\bH}^1_\dR}\ar[r]^-{\varrho}\ar[d]^-{\overline{\iota}\vc\overline{\bH}^1_\dR} & {\overline{\bH}^1_\dR\circ (-,-,-_\red)}\ar[d]^-{\overline{\bH}^1_\dR\vc\iota}\\
 {\overline{\bH}^1_\dR}\ar@{=}[r] & {\overline{\bH}^1_\dR}} 
 \qquad
 \xymatrix{{\fil^1_\sM\circ\bLM}\ar[r]^-{\rho}\ar[d]^-{\iota\vc\bLM} & {\bLM\circ\fil ^1}\ar[d]^-{\bLM\vc\overline{\iota}}\\
 {\bLM}\ar@{=}[r] & {\bLM}}
 \end{equation}
commutative. Note that by \remarkref{RemaMono}, for every object $A$ in $\ECMM_{1}$ the morphism $\overline{\iota}:\fil^1A\ra A$ is a monomorphism.

\subsection{}\label{Fil2} So far we have constructed the first step of the filtration. Let us now construct the second one.
Let $\scr D$ be the full subquiver of $\oMCrv$ with vertices the $k$-curves with modulus $(X,Y,Z)$ such that $Z$ is \emph{reduced}.

We denote by $\sM_{1,\bbQ}^{\inf=0}$ the kernel of the exact functor $\Gr^0_\sM$. This is the category of Laumon 1-isomotives without infinitesimal part and by definition it is the full subcategory of $\sM^a_{1,\bbQ}$ of objects $M$ such that $\Gr_\sM^0(M)=0$ that is such that $\iota_\sM:\fil^1_\sM(M)\ra M$ is an isomorphism. Similarly we denote by $\ECMM^{\inf=0}_{1}$ the kernel of the exact functor
$$\Gr^0:\ECMM_{1}\ra\ECMM_{1} $$
constructed in \pararef{Fil1}. The compatibility given in \eqref{CompFil1} ensures that the functor \eqref{eq:phi} induces an exact functor 
$$\bLM:\ECMM_{1}^{\inf=0}\ra\sM_{1, \bbQ}^{\inf=0}.$$
\begin{prop}\label{univinf=0}
The universal $\bbQ$-linear Abelian category associated with the representation
$$\bH^1_\dR:\scr D^\op\ra\mod(\bbQ)$$
is equivalent to $\ECMM^{\inf=0}_{1}$.
\end{prop}

\begin{proof}
Let us denote by $\scr C$ the associated category and by
$\scr D^{\op}\xra{\overline{\bH}^1_\scr C}\scr C\xra{F_\scr C}\mod(\bbQ)$
the canonical factorization of the restriction of ${\bH}^1_\dR$ to $\scr D^\op$. Since the restriction of $\overline{\bH}^1_\dR$ to $\scr D^\op$ takes its values in the Abelian subcategory $\ECMM^{\inf=0}_{1}$, the universal property of Nori's category ensures the existence of a $\bbQ$-linear exact faithful functor
$$I_\scr C: \scr C\ra\ECMM_{1}^{\inf=0}$$ 
and two invertible natural transformations $\gamma:I_\scr C\circ \overline{\bH}^1_\scr C\ra \overline{\bH}^1_\dR$ and  $\delta:F^a_\dR\circ I_\scr C\ra F_\scr C$ such that the square
$$\xymatrix{{F^a_\dR\circ I_\scr C\circ \overline{\bH}^1_\scr C}\ar[r]^-{F^a_\dR\vc \gamma}\ar[d]^-{\delta\vc \overline{\bH}^1_\scr C} & {F^a_\dR\circ \overline{\bH}^1_\dR}\ar@{=}[d]\\
{F_{\scr C}\circ\overline{\bH}^1_\scr C}\ar@{=}[r] & {\bH^1_\dR}} $$
is commutative. To construct a quasi-inverse to the functor $I_\scr C$ let us go back to the construction of $\fil^1$ in \pararef{Fil1}. Observe that \eqref{MorMCrvD} takes its values in $\scr D$ and that the square \eqref{squareFil1} can be refined in a square
$$\xymatrix{{\oMCrv^\op}\ar[r]^-{\LM}\ar[d]_-{(-,-,-_\red)} & {\sM^a_{1,\bbQ}}\ar[d]^{\fil_\sM^1}\\
{\scr D^\op}\ar[r]^-{\LM|_\scr D} & {\sM^{\inf=0}_{1,\bbQ}.}}$$
By \propositionref{Fonc3} and \propositionref{NatTransf}, this shows the existence of a $\bbQ$-linear exact functor $\fil^1_\scr C:\ECMM_{1}\ra \scr C$ and an invertible natural tranformation $I_\scr C\circ\fil^1_\scr C\ra\fil^1$. 

Let us denote by $I_{\inf=0}$ the inclusion functor of $\ECMM_{1}^{\inf=0}$ into $\ECMM_{1}$. Since $\fil^1\circ I_{\inf=0}$ is isomorphic to the identity, the composition $I_\scr C\circ \fil^1_\scr C\circ I_{\inf=0} $ is isomorphic to the identity. This shows that the faithful functor $I_\scr C$ is an equivalence and that $ \fil^1_\scr C\circ I_{\inf=0} $ is a quasi-inverse.
\end{proof}

Now consider the morphism of quivers
\begin{align*}
\scr D & \ra \scr D\\
(X,Y,Z) & \mapsto (X,Y_\red,Z).
\end{align*}
(This is indeed a morphism because
if $f: X \to X'$ is a morphism of $k$-curves
and if effective divisors $Y \subset X$ and $Y' \subset X'$
satisfy $Y \le f^*Y'$,
then we have $Y_\red \le (f^*Y')_\red \le f^*(Y_\red')$.)
Since the square
$$\xymatrix{{\scr D^\op}\ar[r]^-{\LM}\ar[d]_-{(-,-_\red,-)} & {\sM^{\inf=0}_{1,\bbQ}}\ar[d]^-{\Gr^1_\sM}\\
{\scr D^\op}\ar[r]^-{\LM} & {\sM^{\inf=0}_{1,\bbQ}}} $$
is commutative, \propositionref{Fonc3} and \propositionref{univinf=0} show the existence of a $\bbQ$-linear exact functor
\footnote{Note that the notation might be misleading: $\Gr^1$ is not yet the graded pieces associated to a filtration.}
$\Gr^1:\ECMM^{\inf=0}_{1}\ra\ECMM_{1}^{\inf=0}$ and two invertible natural transformations
$$\rho:\Gr^1_\sM\circ \bLM\ra \bLM\circ \Gr^1\qquad \varrho: \Gr^1\circ \overline{\bH}^1_\dR\ra \overline{\bH}^1_\dR\circ(-,-_\red,-) $$
such that the corresponding diagram as in \eqref{Dia1Fonc3} is commutative.

Note that, for every Laumon 1-isomotives $M$, there is a canonical epimorphism $\fil^1_\sM(M)\ra\Gr^1_\sM(M)$. In particular, if $M$ is without infinitesimal part, there is a canonical epimorphism
$$\pi_\sM:M\ra \Gr^1_\sM(M).$$
Since $Y_\red\leqslant Y$, the identity of $X$ induces an edge from $(X,Y_\red,Z)$ to $(X,Y,Z)$ in $\scr D$.
\begin{rema}
Note that if $Y$ is not reduced, then the identity of $X$ does not define an edge from $(X,Y,Z)$ to $(X,Y_\red,Z_\red)$ in $\oMCrv$. 
This is the main reason for introducing the subquiver $\scr D$.
\end{rema}

This provides a natural transformation
$$\pi_{\scr D}:\Id_{\scr D^\op}\ra (-,-_\red,-) $$
of functors from $\scr D^\op$ with values in $\scr D^\op$. Note that the square
$$\xymatrix{{\LM|_{\scr D}}\ar@{=}[r]\ar[d]_-{\pi_\sM\vc\LM|_\scr D} & {\LM|_\scr D}\ar[d]^-{\LM|_\scr D\vc\pi_{\scr D}}\\
{\Gr^1_\sM\circ\LM|_\scr D}\ar@{=}[r] & {\LM|_\scr D\circ (-,-_\red,-)}} $$
commutes. We may therefore apply \propositionref{NatTransf} to obtain a natural tranformation $\overline{\pi}:\Id\ra\Gr^1$ that makes the squares
$$\xymatrix{{ \overline{\bH}^1_\dR}\ar@{=}[r]\ar[d]^-{\overline{\pi}\vc\overline{\bH}^1_\dR} & {\overline{\bH}^1_\dR}\ar[d]^-{\overline{\bH}^1_\dR\vc\pi_\scr D}\\
 {\Gr^1\circ\overline{\bH}^1_\dR}\ar[r]^-{\varrho} & {\overline{\bH}^1_\dR\circ  (-,-_\red,-)}} 
 \qquad
 \xymatrix{{\bLM}\ar@{=}[r]\ar[d]^-{\pi_\sM\vc\bLM} & {\bLM}\ar[d]^-{\bLM\vc\overline{\pi}}\\
 {\Gr^1_\sM\circ\bLM}\ar[r]^-{\rho} & {\bLM\circ \Gr^1}}
 $$
commutative. Note that in the above squares, all natural transformations are between functors on $\scr D^\op$ or $\ECMM_{1}^{\inf=0}$.  By \remarkref{RemaMono}, for every object $A$ in $\ECMM^{\inf=0}_{1}$ the morphism $\overline{\pi}:A\ra \Gr^1(A)$ is an epimorphism.

Let $A$ be an object in $\ECMM_{1}$. Then $\fil^1(A)$ belongs to $\ECMM_{1}^{\inf=0}$ and we set
\begin{equation}\label{DefFil2}
\fil^2(A):=\Ker\left[\fil^1(A)\ra\Gr^1(\fil^1(A))\right].
\end{equation}
Note that by definition $\Gr^1(A):=\Gr^1(\fil^1(A))$.

\section{Proof of the main theorem}\label{sect:pf-main-thm}

In this section, we assume that $k$ is a number field.
We complete the proof of  \theoremref{MainTheo2}.

\subsection{}
Recall from \propositionref{EmbeddingMotives}
that we have a fully faithful functor 
$I_{\ECM}: \ECM_1^\dR \to \ECMM_1$.
The composition of $I_{\ECM}$ with
$\bLM:\ECMM_{1}\ra\sM^a_{1,\bbQ}$
factors through the category $\sM_{1,\bbQ}$ of Deligne $1$-isomotives.
This induces a functor
$\ECM_1^\dR \to \sM_{1,\bbQ}$ by universality.

\begin{prop}\label{EquivDel2}
The functor $\ECM_1^\dR \to \sM_{1,\bbQ}$ is an equivalence.
%
%
%
\end{prop}

\begin{proof}
This follows from \eqref{eq:anti-eq-ehm-ecm},
 \theoremref{EquivDel}, \propositionref{prop:equiv-betti-dr},
and the Cartier duality for $\sM_{1,\bbQ}$.
%
\end{proof}

Let $\ECMM_{1}^{\uni} $ be the intersection of the kernel of the exact functors $\Gr^0$ and $\Gr^1$ constructed in \pararef{Fil1} and \pararef{Fil2}. An object $A$ in $\ECMM_{1} $ belongs to the full subcategory  $\ECMM_{1}^{\uni}$ if and only if the canonical monomorphism $\fil^2(A)\ra A$ is an isomorphism. Since the functor $\LM$ is compatible with the filtration, it induces a $\bbQ$-linear exact faithful functor
(see \pararef{sect:deligne-1-motives})
$$\LM:\ECMM_{1}^{\uni}\ra\sM_{1,\bbQ}^{\uni}.$$
Note that $\sM_{1,\bbQ}^{\uni}$ is simply the category of unipotent commutative algebraic groups over $k$ and that 
the functor $\sM_{1,\bbQ}^{\uni} \to \mod(k)$ given by
the restriction of the de Rham realization $R_\dR$  is nothing but the functor that associates with a unipotent commutative algebraic $k$-group its Lie algebra and is therefore an equivalence.

\begin{prop}\label{Equivuni}
The functor
$$\LM:\ECMM_{1}^{\uni}\ra\sM_{1,\bbQ}^{\uni}$$
is an equivalence.
\end{prop}

For the proof, we need an elementary lemma.
\begin{lemm}\label{lem:elementary}
For any $\mu \in \bbZ_{>0}$,
$k$ is generated by $\{ a^\mu \mid a \in k \}$
as a $\bbQ$-algebra.
\end{lemm}
\begin{proof}
Write $k=\bbQ(\gamma)$ with $\gamma \in k$.
Then $\gamma$ can be written as a $\bbQ$-linear combination of
$\gamma^\mu, (\gamma+1)^\mu, \dots, (\gamma+\mu-1)^\mu$.
The lemma follows from this.
\end{proof}

\begin{proof}
We define a subquiver $\oMPo$ of $\oMCrv$ as follows.
The vertices are given by 
$\mathrm{P}_n:=(\bbP^1, n[\infty], \emptyset) \in \oMCrv$
for any integer $n \geq 2$.
The edges from $P_n$ to $P_m$ consists of two types:
\begin{align}
\label{auto2}
&\text{Any automorphism of $\bbP^1$ which fixes $\infty$, when $n=m \geq 2$.}
\\
\label{auto1}
&\text{The identity map on $\bbP^1$, when $m \geq n \geq 2.$}
\end{align}
Let $w: \mod(k) \to \mod(\bbQ)$ be the forgetful functor.
Consider the representation $T=w \circ R_{\dR}\circ\LM|_{\oMPo}:\oMPo^\op\ra\mod(\bbQ)$ and its canonical factorization
$$\oMPo^\op\xra{\overline{T}}\scr U\xra{F_T}\mod(\bbQ)$$
where $\scr U=\comod(\eusm{C}_T)$ is Nori's universal category (see \remarkref{RemaNoriCat}). Note that the restriction of the representation $\overline{\bH}^1_\dR$ to the subquiver $\oMPo$ takes its values in $\ECMM_{1}^\uni$. Hence, by the universal property of Nori's construction (see \theoremref{NoriLength}), there exist a $\bbQ$-linear exact faithful functor $U:\scr U\ra\ECMM_{1}^{\uni} $, two invertible natural transformations $\alpha:U\circ \overline{T}\ra\overline{\bH}^1_\dR $ and  $\beta:w \circ R_\dR\circ \bLM\circ U\ra F_T$ such that the diagram
$$\xymatrix{{\oMPo}\ar[rd]_-{\overline{T}}\ar[r]^-{\overline{\bH}^1_\dR} & {\ECMM_{1}^\uni}\ar[r]^-{\bLM} & {\sM_{1,\bbQ}^{\uni}}\ar[r]^-{R_\dR} & {\mod(k)} \ar[r]^-{w}& {\mod(\bbQ)}\\
{}& {\scr U}\ar@/_2em/[rrru]_-{F_T}\ar[u]_-{U} & {} & {}}  $$
is commutative. Since the functor $\bLM $ is faithful, to show the proposition it is enough to show that $R_\dR\circ \bLM\circ U : \scr U \to \mod(k)$ is an equivalence of categories (note that the functor $U$ will then also be an equivalence). It suffices to see that $\eusm{C}_T$ is the $\bbQ$-linear dual of the algebra $k$, and this amounts to check that for every full subquiver $\scr E$ of $\oMPo$ with finitely many objects 
$$\End_\bbQ(T|_\scr E)=k.$$ 
We may assume $\scr E$ of the form $\{\mathrm{P}_2,\ldots,\mathrm{P}_n\}$ for some integer $n\geqslant 2$. Write $\bbP^1=\Proj(k[T,S])$
and put $t=T/S$, $s=S/T$ so that $\infty \in \bbP^1$ is defined by $s=0$.
By \eqref{eq:lie-jacuni} and \eqref{eq:def-lm}, the representation $T$ maps $\mathrm{P}_n$ to the $\bbQ$-vector space $sk[s]/(s^n)$. We compute the action of morphisms on this space in three instances:
\begin{enumerate}
\item[(a)] Let $n \geqslant 2$ and consider the edge $e:\mathrm{P}_n\ra\mathrm{P}_n$ of type \eqref{auto2} given by $t\mapsto at$ where $a$ is a fixed element in $k^\times$. 
Then $T(e)$ is the $k$-linear map represented 
by a diagonal matrix $(a^{-1}, a^{-2}, \dots, a^{1-n})$ with respect to the $k$-basis $\{ s, s^2, \dots, s^{n-1} \}$,
\item[(b)] Let $n \geqslant 2$ and consider the edge $e:\mathrm{P}_n\ra\mathrm{P}_n$  of type \eqref{auto2} given by $t \mapsto t-1$.
Then $T(e)$ maps $s=1/t$ to $1/(t-1)=s+s^2+\dots+s^{n-1}$.
(We will not need to know $T(e)(s^i)$ for $i>1$.)
\item[(c)]
Let $m \geqslant n \geqslant 2$ and consider the edge of type \eqref{auto1}.
Then $T(e)$ is the map
$sk[s]/(s^m) \to sk[s]/(s^n)$
induced by the identity on $sk[s]$.
\end{enumerate}

Let $\alpha$ be an element in $\End_\bbQ(T|_\scr E)$. Then $\alpha$ is given by a  family $(\alpha^{(i)})_{i=2}^n\in\prod_{i=2}^n\End_{\bbQ}(T(\mathrm{P}_i))$ such that for every edge $e:\mathrm{P}_i\ra\mathrm{P}_j$ in $\oMPo$ 
\begin{equation}\label{condend}
\alpha_i\circ T(e)=T(e)\circ\alpha_j.
\end{equation}
We write $\alpha^{(i)}=(\alpha^{(i)}_{\mu \nu})_{\mu, \nu=1, \dots, i}$
with $\alpha^{(i)}_{\mu \nu} \in \End_{\bbQ}(k)
(\cong M_{d}(\bbQ)$ with $d=[k:\bbQ]$).
Let us define a $\bbQ$-algebra embedding $m : k \to \End_{\bbQ}(k)$
by $m(a)(x)=ax ~(a, x \in k)$.

The condition \eqref{condend} for all edges of the form (a) implies that 
\begin{equation}\label{condend2}
 m(a)^{-\mu} \alpha^{(i)}_{\mu \nu} =\alpha^{(i)}_{\mu \nu} m(a)^{-\nu} 
\quad \text{for all $a \in k^\times$.}
\end{equation}
Since $m(a)$ for $a \in \bbQ$ lies in the center of $\End_{\bbQ}(k)$,
\eqref{condend2} applied to, say, $a=2$
yields $\alpha^{(i)}_{\mu \nu}=0$ if $\mu \not= \nu$.
In view of \lemmaref{lem:elementary}, 
it also yields that $\alpha^{(i)}_{\mu \mu}$
belongs to the centralizer of the image of $m$,
which is $k$ itself as $k$ is a maximal commutative subring of $\End_{\bbQ}(k)$.
We write $\alpha^{(i)}_{\mu \mu} = m(a_\mu^{(i)})$ with $a_\mu^{(i)} \in k$.
Applying the condition \eqref{condend} for all edges of the form (b),
we obtain $a_1^{(i)}=a_\mu^{(i)}$ for all $\mu$.
Finally, \eqref{condend} for all edges of type (c) 
yields $a_1^{(i)}=a_1^{(1)}$ for all $i$.
We have shown that $\alpha = a_1^{(1)} \in k$.
This completes the proof.
\end{proof}

Let $\ECMM_{1}^{\inf}$ be the kernel of the exact functor $\fil^1$ constructed in \pararef{Fil1}. 
By a dual argument, we obtain the following proposition.

\begin{prop}\label{Equivinf}
The restriction of the functor
$$\bLM:\ECMM_{1}\ra\sM^a_{1,\bbQ}$$
to the subcategory $\ECMM_{1}^{\inf} $ induces an equivalence of categories between $\ECMM_{1}^{\inf} $ and $\sM_{1, \bbQ}^{\inf}$
(see \pararef{sect:deligne-1-motives}).
\end{prop}

\subsection{} We finally prove our main theorem.

\begin{proof}[Proof of \theoremref{MainTheo2}]
To prove \theoremref{MainTheo2} it is enough to show that we are in a situation where the criterion of \propositionref{AyoubBV} applies. The first condition is obviously satisfied and the second one follows from \propositionref{QuotProp}. It remains to prove that the third condition is also satisfied.

Let $A$ be an object in $\ECMM_{1}$, $M$ be an object in $\sM^a_{1,\bbQ}$ and $u:\bLM(A)\ra M$ be a morphism in $\sM^a_{1,\bbQ}$. By applying the functor $R_\dR$ we get a morphism
$$R_\dR(u):F^a_{\dR}(A)\ra R_{\dR}(M) $$
of $\bbQ$-vector spaces. Note that $u$ induces a commutative diagram in  $\sM^a_{1,\bbQ}$

$$\xy
\xymatrix"*"{{\bLM(\fil^2(A))}\ar[r] &{\bLM(\fil^1(A))}\ar[r]\ar[d] & {\bLM(A)}\ar[d]\\
 & {\bLM(\Gr^1(A))} & {\bLM(\Gr^0(A))}}
 \POS(25,-25) 
 \xymatrix@C=1.5cm{{\fil^2_\sM(M)}\ar[r]\ar@{<-}["*"]  &{\fil^1_\sM(M)}\ar[r]\ar[d]\ar@{<-}["*"]  & {M}\ar[d]\ar@{<-}["*"] \\
 & {\Gr^1_\sM(M)}\ar@{<-}["*"]|(.58)\hole   & {\Gr^0_\sM(M)}\ar@{<-}["*"]|(.58)\hole }
 \endxy$$
Applying the functor $R_\dR$ yields a commutative diagram 
\begin{equation}\label{eq:chase}
\xy
\xymatrix"*"{{F^a_\dR(\fil^2(A))}\ar[r] &{F^a_\dR(\fil^1(A))}\ar[r]\ar[d] & {F^a_\dR(A)}\ar[d]\\
 & {F^a_\dR(\Gr^1(A))} & {F^a_\dR(\Gr^0(A))}}
 \POS(25,-25) 
 \xymatrix{{R_\dR(\fil^2_\sM(M))}\ar[r]\ar@{<-}["*"]^-{v_2}  &{R_\dR(\fil^1_\sM(M))}\ar[r]\ar[d]\ar@{<-}["*"]  & {R_\dR(M)}\ar[d]\ar@{<-}["*"]_-{v} \\
 & {R_\dR(\Gr^1_\sM(M))}\ar@{<-}["*"]|(.58)\hole^(.4){v_1}  & {R_\dR(\Gr^0_\sM(M))}\ar@{<-}["*"]|(.58)\hole_(.7){v_0} }
 \endxy
 \end{equation}
 where we set $v=R_\dR(u)$, $v_0:=R_\dR(\Gr^0_\sM(u))$, $v_1:=R_\dR(\Gr^1_\sM(u)) $ and $v_2:=R_\dR(\fil^2_\sM(u))$  to simplify notations.

By construction of the category $\ECMM_{1}$ (see \remarkref{RemaNoriCat}), there exists a finite subquiver $\scr E$ of $\oMCrv$ such that in the diagram
$$\xymatrix{{F^a_\dR(\fil^2(A))}\ar[r] &{F^a_\dR(\fil^1(A))}\ar[r]\ar[d] & {F^a_\dR(A)}\ar[d]\\
 & {F^a_\dR(\Gr^1(A))} & {F^a_\dR(\Gr^0(A))}}$$
all objects are canonically endowed with a $\End(H^1_\dR|_\scr E)$-module structure and all morphisms are $\End(H^1_\dR|_\scr E)$-linear. Using \propositionref{Equivuni}, \propositionref{Equivinf} and \propositionref{EquivDel2}, by allowing $\scr E$ to be bigger, we may assume that the kernels of the maps $v_0,v_1,v_2$ are sub-$\End(T|_\scr E)$-modules. An easy diagram chase in \eqref{eq:chase} shows that the kernel of $v$ is a sub-$\End(T|_\scr E)$-module of $F^a_\dR$ as well. This concludes the proof.
\end{proof}

\end{document}